\documentclass[11pt]{article}
\usepackage{amscd, amsmath, amssymb, amsthm}
\usepackage[all,cmtip]{xy}
\usepackage[pagebackref]{hyperref}

\title{Hodge theory of classifying stacks}
\author{Burt Totaro}
\date{  }

\def\Z{\text{\bf Z}}
\def\Q{\text{\bf Q}}

\def\C{\text{\bf C}}

\def\F{\text{\bf F}}
\def\N{\text{\bf N}}

\def\arrow{\rightarrow}
\def\inj{\hookrightarrow}
\def\imp{\Rightarrow}

\def\g{\mathfrak{g}}
\def\p{\mathfrak{p}}
\def\l{\mathfrak{l}}
\def\t{\mathfrak{t}}
\def\uu{\mathfrak{u}}
\def\gl{\mathfrak{gl}}
\def\sl{\mathfrak{sl}}
\def\sp{\mathfrak{sp}}
\def\so{\mathfrak{so}}
\DeclareMathOperator{\Ext}{Ext}
\DeclareMathOperator{\Hom}{Hom}
\DeclareMathOperator{\cha}{char}
\DeclareMathOperator{\et}{et}
\DeclareMathOperator{\dR}{dR}
\DeclareMathOperator{\height}{ht}
\DeclareMathOperator{\ad}{ad}

\DeclareMathOperator{\Spin}{Spin}
\DeclareMathOperator{\Spec}{Spec}
\DeclareMathOperator{\Ho}{H}
\def\Sp{Sp}
\def\Sq{Sq}

\def\rad{\text{rad}}
\def\Cech{C}

\begin{document}
\maketitle
\newtheorem{theorem}{Theorem}[section]
\newtheorem{corollary}[theorem]{Corollary}
\newtheorem{proposition}[theorem]{Proposition}
\newtheorem{lemma}[theorem]{Lemma}

\theoremstyle{definition}
\newtheorem{definition}[theorem]{Definition}
\newtheorem{example}[theorem]{Example}

\theoremstyle{remark}
\newtheorem{remark}[theorem]{Remark}

This paper creates a correspondence between
the representation theory of algebraic groups and the topology
of Lie groups. In more detail, we compute the Hodge and de Rham
cohomology of the classifying space $BG$ (defined as etale cohomology
on the algebraic stack $BG$) for reductive groups $G$
over many fields, including fields of small characteristic. These
calculations have a direct relation with representation theory, yielding
new results there. Eventually, $p$-adic Hodge theory
should provide a more subtle
relation between these calculations in positive characteristic
and torsion in the cohomology of the classifying space $BG_{\C}$.

For the representation theorist, this paper's interpretation
of certain Ext groups (notably for reductive groups in positive
characteristic) as Hodge cohomology groups
suggests spectral sequences that were not obvious in terms of Ext groups
(Proposition \ref{ls}).
We apply these spectral sequences to compute
Ext groups in new cases. The spectral sequences form
a machine that can lead to further calculations.

One main result is an isomorphism between the Hodge cohomology
of the classifying stack $BG$
and the cohomology of $G$ as an algebraic group with coefficients
in the ring $O(\g)=S(\g^*)$ of polynomial functions
on the Lie algebra $\g$ (Theorem \ref{iso}):
$$H^i(BG,\Omega^j)\cong H^{i-j}(G, S^j(\g^*)).$$
This was shown by Bott over a field of characteristic 0 \cite{Bott},
but in fact
the isomorphism holds integrally. More generally, we give an analogous
description of the equivariant Hodge cohomology of an affine scheme
(Theorem \ref{equivariant}). This
was shown by Simpson and Teleman in characteristic 0
\cite[Example 6.8(c)]{ST}.

Using that isomorphism, we improve the known results on the cohomology
of the representations $S^j(\g^*)$. Namely, by Andersen, Jantzen,
and Donkin, we have $H^{>0}(G,O(\g))=0$ for a reductive group $G$
over a field of characteristic $p$ if $p$ is a ``good prime'' for $G$
\cite[Proposition and proof of Theorem 2.2]{Donkin},
\cite[II.4.22]{Jantzen}. We strengthen that
to an ``if and only if'' statement
(Theorem \ref{non-torsion}):

\begin{theorem}
Let $G$ be a reductive group over a field $k$ of characteristic $p\geq 0$.
Then $H^{>0}(G,O(\g))=0$ if and only if $p$ is not a torsion prime
for $G$.
\end{theorem}

For example, this cohomology vanishing holds for
every symplectic group $\Sp(2n)$
in characteristic 2 and for the exceptional group $G_2$
in characteristic 3;
these are ``bad primes'' but not torsion primes.

Finally, we begin the problem of computing the Hodge cohomology
and de Rham cohomology of $BG$, especially at torsion primes.
At non-torsion primes, we have a satisfying result, proved
using ideas from topology
(Theorem \ref{integral}):

\begin{theorem}
Let $G$ be a split reductive group over $\Z$, and let $p$ be a non-torsion
prime for $G$. Then Hodge cohomology $H^*_{\Ho}(BG/\Z)$
and de Rham cohomology $H_{\dR}^*(BG/{\Z})$, localized at $p$,
are polynomial rings on generators of degrees equal
to 2 times the fundamental degrees of $G$.
These graded rings are isomorphic to the cohomology of the topological
space $BG_{\C}$ with $\Z_{(p)}$ coefficients.
\end{theorem}

At torsion primes $p$, it is an intriguing question how the de Rham
cohomology of $BG_{\F_p}$ is related to the mod $p$ cohomology
of the topological space $BG_{\C}$. We show that these graded rings
are isomorphic for $G=SO(n)$ with $p=2$ (Theorem \ref{son}).
On the other hand, we find that
$$\dim_{\F_2} H^{32}_{\dR}(B\Spin(11)/\F_2)
>\dim_{\F_2} H^{32}(B\Spin(11)_{\C},\F_2)$$
(Theorem \ref{spin}).
It seems that no existing results on integral $p$-adic Hodge theory
address the relation between these two rings (because the stack $BG$
is not proper over $\Z$), but the theory may soon reach that point.
In particular, the results of Bhatt-Morrow-Scholze suggest that
the de Rham cohomology $H^i_{\dR}(BG/\F_p)$ may always be an upper bound
for the mod $p$ cohomology of the topological space $BG_{\C}$ \cite{BMS}.

This work was supported by NSF grant DMS-1303105. Bhargav Bhatt
convinced me to change some definitions in an earlier version
of this paper: Hodge and de Rham cohomology of a smooth stack
are now defined as etale cohomology. Thanks to Johan de Jong,
Eric Primozic, and Rapha\"el Rouquier for their comments.
Finally, I am grateful
to Jungkai Chen for arranging my visit to National Taiwan University,
where this work was completed.

\section{Notation}
\label{notation}

The {\it fundamental degrees }of a reductive group $G$ over a field $k$
are the degrees
of the generators of the polynomial ring $S(X^*(T)\otimes_{\Z}\Q)^W$
of invariants under the Weyl group $W$, where $X^*(T)$ is the character group
of a maximal torus $T$. For $k$ of characteristic zero, the fundamental
degrees of $G$ can also be viewed as the degrees of the generators
of the polynomial ring $O(\g)^G$ of invariant
functions on the Lie algebra.
Here are the fundamental degrees of the simple groups
\cite[section 3.7, Table 1]{Humphreysrefl}:

\label{degrees}
$$\begin{array}{cl}
A_l &  2,3,\ldots,l+1\\
B_l & 2,4,6,\ldots,2l\\
C_l & 2,4,6,\ldots,2l\\
D_l & 2,4,6,\ldots,2l-2;l\\
G_2 & 2,6\\
F_4 & 2,6,8,12\\
E_6 & 2,5,6,8,9,12\\
E_7 & 2,6,8,10,12,14,18\\
E_8 & 2,8,12,14,18,20,24,30
\end{array}$$

For a commutative ring $R$ and $j\geq 0$, write $\Omega^j$ for the sheaf
of differential forms over $R$ on any scheme over $R$.
For an algebraic stack $X$ over $R$, $\Omega^j$ is a sheaf
of abelian groups
on the big etale site of $X$. (In particular,
for every scheme $Y$ over $X$ of ``size'' less than a fixed limit ordinal
$\alpha$ \cite[Tag 06TN]{Stacks},
we have an abelian group $\Omega^j(Y/R)$, and these groups form a sheaf in
the etale topology.) We define Hodge
cohomology $H^i(X,\Omega^j)$ to mean the etale cohomology
of this sheaf \cite[Tag 06XI]{Stacks}. In the same way, we define
de Rham cohomology of a stack, $H^i_{\dR}(X/R)$, as etale
cohomology with coefficients in the de Rham complex over $R$.
(If $X$ is an algebraic space,
then the cohomology of a sheaf $F$ on the big etale site
of $X$ coincides with the cohomology of the restriction
of $F$ to the small etale site, the latter being the usual
definition of etale cohomology for algebraic spaces
\cite[Tag 0DG6]{Stacks}.) For example, this gives a definition
of equivariant Hodge or De Rham cohomology, $H^i_G(X,\Omega^j)$
or $H^i_{G,\dR}(X/R)$, as the Hodge
or de Rham cohomology of the quotient stack $[X/G]$.
Essentially the same definition was used for smooth stacks
in characteristic zero
by Teleman and Behrend \cite{Teleman, Behrend}.

This definition of Hodge and de Rham cohomology
is the ``wrong'' thing to consider
for an algebraic stack which is not
smooth over $R$. For non-smooth stacks, it would be better
to define Hodge and de Rham cohomology 
using some version of Illusie and Bhatt's derived de Rham cohomology,
or in other words using the cotangent complex
\cite[section 4]{Bhatt}. In this paper, we will only consider
Hodge and de Rham cohomology for smooth stacks over a commutative ring $R$.
An important example for the paper
is that the classifying stack $BG$ is smooth over $R$
even for non-smooth group schemes $G$ \cite[Tag 075T]{Stacks}:

\begin{lemma}
\label{flatBG}
Let $G$ be a group scheme which is flat
and locally of finite presentation over a commutative ring $R$. Then
the algebraic stack $BG$ is smooth over $R$. More generally,
for a smooth algebraic
space $X$ over $R$ on which $G$ acts, the quotient stack $[X/G]$
is smooth over $R$.
\end{lemma}

Let $X$ be an algebraic stack over $R$, and let $U$ be an algebraic space
with a smooth surjective morphism to $X$. The \u{C}ech
construction $C(U/X)$ means the
simplicial algebraic space:
$$\xymatrix@C-10pt@R-10pt{
U \ar@/_1pc/@<-0.0ex>[r] 
& U\times_{X}U \ar@<-0.5ex>[l]\ar@<0.5ex>[l]
\ar@/_1pc/@<-0.0ex>[r] \ar@/_1pc/@<-1.0ex>[r]
& U\times_{X}U\times_{X}U \ar@<-1ex>[l]\ar@<0ex>[l]\ar@<1ex>[l]\cdots
}$$

For any sheaf $F$ of abelian groups on the big etale site of $X$,
the etale cohomology of $X$ with coefficients in $F$ can be identified with
the etale cohomology of the simplicial algebraic space $C(U/X)$
\cite[Tag 06XJ]{Stacks}.
In particular, there is a spectral sequence:
$$E_1^{ij}=H^j_{\et}(U^{i+1}_X,F)\imp H^{i+j}_{\et}(X,F).$$

Write $H^i_{\Ho}(X/R)=\oplus_j H^j(X,\Omega^{i-j})$
for the Hodge cohomology of an algebraic stack $X$ over $R$,
graded by total degree.

Let $G$ be a group scheme which is flat and locally of finite
presentation over a commutative ring $R$.
Then the Hodge cohomology of the stack $BG$
can be viewed, essentially by definition,
as the ring of characteristic classes in Hodge cohomology
for principal $G$-bundles (in the fppf topology). Concretely,
for any scheme $X$ over $R$, a principal $G$-bundle
over $X$ determines a morphism $X\arrow BG$ of stacks
over $R$ and hence a pullback homomorphism
$$H^i(BG,\Omega^j)\arrow H^i(X,\Omega^j).$$
Note that for a scheme $X$ over $R$, $H^i(X,\Omega^j)$ can be computed
either in the Zariski or in the etale topology, because the sheaf
$\Omega^j$ (on the small etale site of $X$) is quasi-coherent
\cite[Tag 03OY]{Stacks}.

For any scheme $X$ over a commutative ring $R$, there is a simplicial
scheme $EX$ whose space $(EX)_n$
of $n$-simplices is $X^{\{0,\ldots,n\}}=X^{n+1}$
\cite[6.1.3]{Deligne}.
For a group scheme $G$ over $R$,
the simplicial scheme $BG$ over $R$ is defined as
the quotient of the simplicial scheme $EG$ by the free
left action of $G$:
$$\xymatrix@C-10pt@R-10pt{
\Spec(R) \ar@/_1pc/@<-0.0ex>[r]
& G \ar@<-0.5ex>[l]\ar@<0.5ex>[l]
\ar@/_1pc/@<-0.0ex>[r] \ar@/_1pc/@<-1.0ex>[r]
& G^2 \ar@<-1ex>[l]\ar@<0ex>[l]\ar@<1ex>[l]\cdots
}$$

If $G$ is smooth over $R$, then Hodge cohomology $H^i(BG,\Omega^j)$
as defined above can be identified with the cohomology of the
simplicial scheme $BG$, because this simplicial scheme
is the \u{C}ech simplicial scheme associated to the smooth
surjective morphism $\Spec(R)\arrow BG$. For $G$ not smooth,
one has instead to use the \u{C}ech simplicial scheme associated
to a smooth presentation of $BG$.
See for example the calculation of the Hodge cohomology of $B\mu_p$
in characteristic $p$, Proposition \ref{mup}.

It is useful that we can compute Hodge cohomology
via any smooth presentation of a stack. For example, let $H$ be a closed
subgroup scheme of a smooth group scheme $G$ over a commutative ring $R$,
and assume
that $H$ is flat and locally of finite presentation over $R$.
Then $G/H$ is an algebraic space with a smooth surjective morphism
$G/H\arrow BH$ over $R$, and so we can compute the Hodge
cohomology of the stack $BH$ using the associated \u{C}ech simplicial
algebraic space.
Explicitly, that is the simplicial algebraic space $EG/H$, and so we have:

\begin{lemma}
\label{EG}
$$H^i(BH,\Omega^j)\cong H^i(EG/H,\Omega^j).$$
\end{lemma}

Note that the cohomology theories we are considering
are not $A^1$-homotopy invariant.
Indeed, Hodge cohomology is usually not the same for a scheme $X$
as for $X\times A^1$, even over a field of characteristic zero.
For example, $H^0(\Spec(k),O)=k$, whereas $H^0(A^1_k,O)$ is the polynomial
ring $k[x]$.
In de Rham cohomology, $H^0_{\dR}(A^1/k)$ is just $k$
if $k$ has characteristic zero, but it is $k[x^p]$
if $k$ has characteristic $p>0$.

\section{Equivariant Hodge cohomology and functions on
the Lie algebra}

In this section, we identify the Hodge cohomology of a quotient
stack with the cohomology of an explicit complex of vector
bundles (Theorem \ref{equivariant}). As a special case,
we relate the Hodge cohomology of a classifying
stack $BG$ to the cohomology of $G$ as an algebraic group
(Corollary \ref{isosmooth}). In this section, we assume $G$ is smooth.
Undoubtedly, various generalizations of the statements here
are possible. In particular, we will give an analogous description
of the Hodge cohomology of $BG$ for a non-smooth group $G$
in Theorem \ref{iso}.

The main novelty is that these results hold in any characteristic.
In particular, Theorem \ref{equivariant} was proved
in characteristic zero by Simpson and Teleman
\cite[Example 6.8(c)]{ST}.

\begin{theorem}
\label{equivariant}
Let $G$ be a smooth affine group scheme
over a commutative ring $R$. Let $G$ act on a smooth
affine scheme $X$ over $R$. Then there is a canonical
isomorphism
$$H^i_G(X,\Omega^j)\cong H^i_G(X, \Lambda^j L_{[X/G]}),$$
where $\Lambda^j L_{[X/G]}$ is the complex of $G$-equivariant
vector bundles on $X$, in degrees 0 to $j$:
$$0\arrow \Omega^j_X \arrow \Omega^{j-1}_X\otimes \g^*\arrow
\cdots \arrow S^j(\g^*)\arrow 0,$$
associated to the map $\g\arrow TX$.
\end{theorem}

This isomorphism expresses the cohomology
over $[X/G]$ of the ``big sheaf'' $\Omega^j$, which is
not a quasi-coherent sheaf on $[X/G]$,
in terms of the cohomology of a complex
of quasi-coherent sheaves on $[X/G]$. (Here differentials are over $R$
unless otherwise stated. The sheaf $\Omega^j$
on the big etale site of $[X/G]$
is not quasi-coherent for $j>0$ because, for a morphism
$f\colon Y\arrow Z$ of schemes over $[X/G]$,
the pullback map $f^*\Omega^j_{Z/R}\arrow \Omega^j_{Y/R}$
need not be an isomorphism.) Theorem \ref{equivariant} is useful
already for $X=\Spec(R)$, where it gives the following result,
proved in characteristic zero by Bott \cite{Bott}.

\begin{corollary}
\label{isosmooth}
Let $G$ be a smooth affine group scheme
over a commutative ring $R$. Then there is a canonical
isomorphism
$$H^i(BG,\Omega^j)\cong H^{i-j}(G, S^j(\g^*)).$$
\end{corollary}

The group on the left is an etale cohomology group of the algebraic stack
$BG$ over $R$, as discussed in section \ref{notation}. On the right
is the cohomology of $G$ as an algebraic group,
defined by
$H^i(G,M)=\Ext^i_G(R,M)$ for a $G$-module $M$
\cite[section 4.2]{Jantzen}.

\begin{proof}
(Corollary \ref{isosmooth}) This follows from Theorem \ref{equivariant}
applied to the stack $BG=[\Spec(R)/G]$. The deduction uses two facts.
First, a quasi-coherent sheaf
on $BG$ is equivalent to a $G$-module \cite[Tag 06WS]{Stacks}.
Second, for a $G$-module $M$, the cohomology of
the corresponding quasi-coherent sheaf on the big etale
site of $BG$ coincides with its cohomology as a $G$-module,
$H^*(G,M)$, since both are computed by the same \u{C}ech complex
(section \ref{notation} for the sheaf,
\cite[Proposition 4.16]{Jantzen} for the module).
\end{proof}

\begin{proof}
(Theorem \ref{equivariant})
The adjoint representation of $G$ on $\g$ determines a $G$-equivariant
vector bundle $\g$ on $X$.
The action of $G$ on $X$ gives a morphism $\Omega^1_X\arrow \g^*$
of $G$-equivariant quasi-coherent sheaves (in fact, vector bundles)
on $X$. Consider
these equivariant sheaves as quasi-coherent sheaves
on $[X/G]$, according to \cite[Tag 06WS]{Stacks}.

We will define a map from the complex $\Omega^1_X\arrow \g^*$
of quasi-coherent sheaves on $[X/G]$ (in degrees 0 and 1)
to the sheaf $\Omega^1$, in the derived
category $D([X/G]_{\et},O_{[X/G]})$ of $O_{[X/G]}$-modules on
the big etale site $[X/G]_{\et}$.
To do this, define another sheaf
$S$ on the big etale site of $[X/G]$ by: for a scheme $U$ over $[X/G]$,
let $E=U\times_{[X/G]}X$ (so that $E\arrow U$ is a principal $G$-bundle),
and define $S(U)=H^0(E,\Omega^1)^G$. There is a short
exact sequence
$$0\arrow H^0(U,\Omega^1)\arrow H^0(E,\Omega^1)^G\arrow H^0(U,\g^*)
\arrow 0$$
for each affine scheme $U$ over $[X/G]$.
(Note that the principal $G$-bundle $E$ over $U$ together
with the adjoint action of $G$ on $\g^*$ determines a vector bundle
which we call $\g^*$ on $U$.) So we have an exact sequence
$$0\arrow \Omega^1\arrow S\arrow \g^*\arrow 0$$
of sheaves on the big etale site of $[X/G]$ (where $\g^*$ is
a vector bundle on $[X/G]$).

Thus the sheaf $\Omega^1$ on $[X/G]$ is isomorphic in the derived
category to the complex $S\arrow \g^*$ (in cohomological degrees 0 and 1)
on $[X/G]$. Therefore, to produce
the map in $D([X/G]_{\et},O_{[X/G]})$ promised above, it suffices to define 
a map of complexes of sheaves on $[X/G]_{\et}$:
$$\xymatrix@C-10pt@R-10pt{
0 \ar[r] & \Omega^1_X\ar[r]\ar[d] & \g^* \ar[r]\ar[d] & 0\\
0 \ar[r] & S\ar[r] & \g^* \ar[r] & 0.
}$$
(As above, $\g^*$ denotes the vector bundle on $[X/G]$
associated to the representation of $G$ on $\g^*$,
and $\Omega^1_X$ denotes the vector bundle on $[X/G]$
corresponding to the $G$-equivariant vector bundle of the same
name on $X$.) It is now easy to produce the map of complexes:
for any scheme $U$ over $[X/G]$, with associated principal
$G$-bundle $E\arrow U$ and $G$-equivariant morphism
$h\colon E\arrow X$, the map from $\Omega^1_X(U)=H^0(E,h^*\Omega^1_X)^G$
to $S(U)=H^0(E,\Omega^1_E)^G$ is the pullback, and the map
from $\g^*$ to itself is the identity.

For any $j\geq 0$, taking the $j$th derived exterior power over $O_{[X/G]}$
of this map
of complexes gives a map from the Koszul complex
$$0\arrow \Omega^j_X \arrow \Omega^{j-1}_X\otimes \g^*\arrow
\cdots \arrow S^j(\g^*)\arrow 0$$
(in degrees 0 to $j$) of vector bundles on $[X/G]$
to the big sheaf $\Omega^j$, in $D([X/G]_{\et},O_{[X/G]})$.
(The description of the derived exterior power of a 2-term complex
of flat modules as a Koszul complex
follows from Illusie
\cite[Proposition II.4.3.1.6]{Illusie1}, by the same argument
used for derived divided powers in \cite[Lemme VIII.2.1.2.1]{Illusie2}.)
We want to show that this map of complexes induces an isomorphism
on cohomology over $[X/G]$.

By the exact sequence above for the big sheaf $\Omega^1$
on $[X/G]$, we can identify the big sheaf $\Omega^j$
in the derived category with a similar-looking
Koszul complex:
$$0\arrow \Lambda^j(S) \arrow \Lambda^{j-1}(S)\otimes \g^*\arrow
\cdots \arrow S^j(\g^*)\arrow 0.$$
We want to show that the obvious map from the Koszul complex
of vector bundles (in the previous paragraph)
to this complex of big sheaves induces
an isomorphism on cohomology over $[X/G]$. It suffices to show
that for each $0\leq i\leq j$, the map
$$H^*_G(X,\Omega^i_X\otimes S^{j-i}(\g^*))\arrow
H^*_G(X,\Lambda^i(S)\otimes S^{j-i}(\g^*))$$
is an isomorphism.

By section \ref{notation}, we can compute both of these cohomology groups
on the \u{C}ech simplicial space associated to the smooth
surjective morphism $X\arrow [X/G]$. This simplicial space can be written
as $(X\times EG)/G$, where all products are over $R$:
$$\xymatrix@C-10pt@R-10pt{
X \ar@/_1pc/@<-0.0ex>[r]
& X\times G \ar@<-0.5ex>[l]\ar@<0.5ex>[l]
\ar@/_1pc/@<-0.0ex>[r] \ar@/_1pc/@<-1.ex>[r]
& X\times G^2 \ar@<-1ex>[l]\ar@<0ex>[l]\ar@<1ex>[l]\cdots,
}$$

Since $X$ is affine, all the spaces in this simplicial
space are affine schemes. Therefore, for any $0\leq i\leq j$,
$H^*_G(X,\Omega^i_X\otimes S^{j-i}(\g^*))$
is the cohomology
of the complex of $H^0$ of the sheaves $\Omega^i_X\otimes O_{EG}\otimes 
S^{j-i}(\g^*)$ over the spaces making up $(X\times EG)/G$.
Likewise, $H^*_G(X,\Lambda^i(S)\otimes S^{j-i}(\g^*))$
is the cohomology of the complex of $H^0$ of the sheaves
$\Lambda^i(S)\otimes 
S^{j-i}(\g^*)$ over the spaces making up $(X\times EG)/G$.

Both of these complexes are spaces of $G$-invariants
of analogous complexes of $H^0$ of sheaves over the spaces making up
$X\times EG$. Moreover, all of these $G$-modules are induced
from representations of the trivial group, because $X\times G^{r+1}
\arrow (X\times G^{r+1})/G$ is a $G$-torsor with a section
for each $r\geq 0$. Indeed, a choice of section of this $G$-torsor
trivializes the torsor, and so the group of
sections of a $G$-equivariant sheaf
of $X\times G^{r+1}$ is the subspace of invariants tensored
with $O(G)$, as a $G$-module. (Note that trivializations
of these $G$-torsors cannot be made compatible with the face maps
of the simplicial space, in general.) 
And every tensor product $O(G)\otimes_R M$ for a $G$-module $M$
is injective as a $G$-module \cite[Proposition 3.10]{Jantzen}.
It follows that $H^i(G,O(G)\otimes_R M)=0$ for $i>0$
\cite[Lemma I.4.7]{Jantzen}.

Therefore, to show that the map of complexes
of $G$-invariants in the previous paragraph
is a quasi-isomorphism (as we want), it suffices
to show that the map of complexes of $H^0$ over $X\times EG$
is a quasi-isomorphism. And for that, we can forget about the $G$-action.
That is, we want to show that the map of complexes with $r$th
term (for $r\geq 0$)
$$H^0(X\times G^{r+1}, \Omega^i_X\otimes O_{G^{r+1}}\otimes 
S^{j-i}(\g^*))\arrow H^0(X\times G^{r+1}, \Omega^i_{X\times G^{r+1}}\otimes
S^{j-i}(\g^*))$$
is a quasi-isomorphism.

We can write $\Omega^i_{X\times G^{r+1}}$ as the direct sum
$\oplus_{l=0}^i\Omega^{i-l}_X\otimes \Omega^l_{G^{r+1}}$.
Moreover, this splitting is compatible with pullback along
the face maps of the simplicial scheme $X\times EG$. So the
map of complexes above
is the inclusion of a summand (corresponding to $l=0$).
It remains to show that for every $0<l\leq i$, the $l$th summand is a complex
with cohomology zero. Its $r$th term is
$$\Omega^{i-l}(X)\otimes_R \Omega^l(G^{r+1})
\otimes_R S^{j-i}(\g^*).$$
To analyze its cohomology, we use the well-known ``contractibility''
of $EG$, in the following form:

\begin{lemma}
\label{EGhomotopy}
Let $Y$ be an affine scheme over a ring $R$ with $Y(R)$ not empty.
For any sheaf $M$ of abelian groups on the big etale site of $R$,
the cohomology of the simplicial scheme $EY$ over $R$ coincides with
the cohomology of $\Spec(R)$:
$$H^i(EY,M)\cong H^i(R,M).$$
\end{lemma}

\begin{proof}
Since $Y$ is affine, $H^*(EY,M)$ is the cohomology of an explicit complex
$$0\arrow M(Y)\arrow M(Y^2)\arrow \cdots.$$
Choosing a point $1\in Y(R)$ gives an explicit chain homotopy
from the identity map to the complex $M(R)$ in degree 0:
$$(F\varphi)(y_0,\ldots,y_{r-1})=\varphi(1,y_0,\ldots,y_{r-1})$$
for $\varphi\in M(Y^{r+1})$ with $r\geq 0$.
\end{proof}

Returning to the proof of Theorem \ref{equivariant}:
we want to show that for $l>0$, the complex with $r$th term
$$\Omega^{i-l}(X)\otimes_R \Omega^l(G^{r+1})
\otimes_R S^{j-i}(\g^*)$$
has zero cohomology.
By Lemma \ref{EG} (applied
to the sheaf $\Omega^l$ on the big etale site of $R$
and the simplicial scheme $EG$),
the complex above has cohomology
equal to $\Omega^{i-l}(X)\otimes \Omega^l(\Spec R)\otimes
S^{j-i}(\g^*)$ in degree 0 and zero in other degrees.
Since $l>0$, the cohomology in degree 0 also vanishes.
The proof is complete.
\end{proof}

The argument works verbatim to prove a twisted version
of Corollary \ref{isosmooth},
where the sheaf $\Omega^j$ on $BG$ is tensored with the vector bundle
associated to any $G$-module. The generalization will not be needed
in this paper, but we state it for possible later use.

\begin{theorem}
\label{isotwist}
Let $G$ be a smooth affine group scheme
over a commutative ring $R$. 
Let $M$ be a $G$-module that is flat over $R$.
Then there is a canonical isomorphism
$$H^i(BG,\Omega^j\otimes M)\cong H^{i-j}(G, S^j(\g^*)\otimes M).$$
\end{theorem}

\section{Flat group schemes}
\label{coLie}

We now describe the Hodge cohomology of the classifying stack
of a group scheme $G$
which need not be smooth, generalizing Corollary \ref{isosmooth}.
The analog of the co-Lie algebra
$\g^*$ in this generality is the co-Lie complex $l_G$ in the derived
category of $G$-modules,
defined by Illusie \cite[section VII.3.1.2]{Illusie2}. Namely,
$l_G$ is the pullback of the cotangent complex of $G\arrow \Spec(R)$
to $\Spec(R)$, via the section $1\in G(R)$. (The cotangent
complex $L_{X/Y}$ of a morphism $X\arrow Y$ of schemes
is an object of the quasi-coherent derived
category of $X$;
if $X$ is smooth over $Y$, then $L_{X/Y}$ is the sheaf $\Omega^1_{X/Y}$.)
The cohomology of $l_G$ in degree 0 is the $R$-module $\omega^1_G$,
the restriction of $\Omega^1_G$ to the identity $1\in G(R)$;
thus $\omega^1_G$ is the co-Lie algebra $\g^*$ if $G$ is smooth over $R$.
The complex $l_G$ has zero cohomology except in cohomological
degrees $-1$ and 0. If $G$
is smooth, then $l_G$ has cohomology concentrated in degree 0.

\begin{theorem}
\label{iso}
Let $G$ be a flat affine group scheme of finite presentation
over a commutative ring $R$. Then there is a canonical
isomorphism
$$H^i(BG,\Omega^j)\cong H^{i-j}(G, S^j(l_G)).$$
\end{theorem}

\begin{proof}
As discussed in section \ref{notation}, we can compute $H^*(BG,\Omega^j)$
as the etale cohomology with coefficients in $\Omega^j$
of the \u{C}ech simplicial space associated to any smooth
algebraic space $U$ over $R$ with a smooth surjective morphism from $U$
to the stack $BG$. The assumption on $G$ implies that $BG$ is
a quasi-compact algebraic stack over $R$, and so there is an affine
scheme $U$ with a smooth surjective morphism $U\arrow BG$
\cite[Tags 06FI and 04YA]{Stacks}. By Lemma \ref{flatBG},
$BG$ is smooth over $R$, and so $U$ is smooth over $R$.
Let $E=U\times_{BG}\Spec(R)$; then $E$ is a smooth $R$-space
with a free $G$-action such that $U=E/G$. Also, $E$ is affine
because $U$ and $G$ are affine.

By section \ref{notation}, $H^*(BG,\Omega^j)$
is the etale cohomology with coefficients in $\Omega^j$ of the simplicial
algebraic space $EE/G$:
$$\xymatrix@C-10pt@R-10pt{
E/G \ar@/_1pc/@<-0.0ex>[r] 
& E^2/G \ar@<-0.5ex>[l]\ar@<0.5ex>[l]
\ar@/_1pc/@<-0.0ex>[r] \ar@/_1pc/@<-1.0ex>[r]
& E^3/G \ar@<-1ex>[l]\ar@<0ex>[l]\ar@<1ex>[l]\cdots
}$$

By the properties of $E$ and $Y$ above,
$E^{n+1}/G$ is an affine scheme for all $n\geq 0$. Since
$H^*(BG,\Omega^j)$ is the cohomology with coefficients in $\Omega^j$
of the simplicial scheme $EE/G$, this is the cohomology
of the cochain complex
$$0\arrow \Omega^j(E/G)\arrow \Omega^j(E^2/G)\arrow \cdots.$$

As in the proof of Theorem \ref{equivariant}, this complex
is the $G$-invariants of the complex
$$0\arrow H^0(E,\pi^*(\Omega^j_{E/G}))\arrow H^0(E^2,
\pi^*(\Omega^j_{E^2/G}))\arrow \cdots,$$
where we write $\pi$ for the morphism $E^{n+1}\arrow E^{n+1}/G$
for any $n\geq 0$. 

For any smooth $R$-scheme $X$ with a free action of $G$,
there is a canonical exact triangle in the quasi-coherent
derived category of $G$-equivariant
sheaves on $X$:
$$\pi^*(\Omega^1_{X/G})\arrow \Omega^1_X\arrow l_G,$$
where we write $l_G$ for the pullback of the co-Lie complex
$l_G$ from the stack $BG$ over $R$
to $X$. To deduce this from Illusie's results on the cotangent complex
$L_{X/Y}$, let $Y=X/G$ and $S=\Spec(R)$, and use
the transitivity exact triangle for $X\arrow Y\arrow S$
in the derived category of $X$ \cite[II.2.1.5.2]{Illusie1}:
$$\pi^*L_{Y/S}\arrow L_{X/S}\arrow L_{X/Y}.$$
Since $X$ is smooth over $S$, so is $Y$ (even though $G$ need not be);
so $L_{Y/S}\cong \Omega^1_{Y/S}$ and $L_{X/S}\cong \Omega^1_{X/S}$.
Also, since $X\arrow Y$ is a $G$-torsor in the fppf topology,
$L_{X/Y}$ is the pullback of an object $l_{X/Y}$ on $Y$
\cite[VII.2.4.2.8]{Illusie2}. Furthermore,
$l_{X/Y}$ in the fppf topology is the pullback of $l_G$
via the morphism from $Y$ to the stack $BG$ corresponding to
the $G$-torsor $X\arrow Y$
\cite[VII.3.1.2.6]{Illusie2}.

Applying this to $E^{n+1}/G$ for any $n\geq 0$,
we get an exact triangle
$$\pi^*(\Omega^1_{EE/G})\arrow \Omega^1_{EE}\arrow l_G$$
in $D_G(EE)$, or equivalently
$$l_G[-1]\arrow \pi^*(\Omega^1_{EE/G})\arrow \Omega^1_{EE}.$$
It follows that for any $j\geq 0$,
$\pi^*(\Omega^j_{EE/G})$ has a filtration in the derived category
with quotients $\pi^*(\Omega^{j-m}_{EE/G})\otimes \Lambda^{m}(l_G[-1])$
for $m=0,\ldots,j$.

If $E(R)$ is nonempty, then $H^i(EE,\Omega^j)\cong H^i(\Spec(R),\Omega^j)$,
by Lemma \ref{EGhomotopy}.
That group is zero unless $i=j=0$, in which case it is $R$. By faithfully
flat descent, the same conclusion holds under our weaker assumption
that $E\arrow \Spec(R)$ is smooth and surjective. Therefore,
in the filtration above, all objects but one have zero cohomology in all
degrees over $EE$. We deduce that the homomorphism
$$H^i(EE,\Lambda^j(l_G[-1]))\arrow H^i(EE,\pi^*(\Omega^j_{EE/G}))$$
is an isomorphism of $G$-modules for all $i$.
By Illusie's ``d\'ecalage'' isomorphism
\cite[Proposition I.4.3.2.1(i)]{Illusie1}, we can
write $S^j(l_G)[-j]$ instead of $\Lambda^j(l_G[-1])$.

The cochain complex $O(EE)$ has cohomology $R$ in degree 0
and 0 otherwise, by Lemma \ref{EGhomotopy} again. So the
complex of global sections of the trivial vector bundle $S^j(l_G)$
over $EE$ is isomorphic, in the derived category of $G$-modules,
to the complex of $G$-modules $S^j(l_G)$. We conclude
that the complex of sections of $\pi^*(\Omega^j_{EE/G})$ over $EE$
is isomorphic
to $S^j(l_G)[-j]$ in the derived category of $G$-modules.

Finally, we observe that each $G$-module in this complex,
$$M:=H^0(E^{n+1},\pi^*(\Omega^j_{E^{n+1}/G}))$$
for $n\geq 0$,
is acyclic (meaning that $H^{>0}(G,M)=0$). More generally,
for any affine $R$-scheme $Y$ with a free $G$-action such
that $Y/G$ is affine, and 
any quasi-coherent sheaf $F$ on $Y/G$,
$M:=H^0(Y,\pi^*F)$ is acyclic. Indeed, this holds if 
$Y\arrow Y/G$ is a trivial $G$-bundle, since then
$M=O(G)\otimes F$ and so $M$ is acyclic \cite[Lemma 4.7]{Jantzen}.
We can prove acyclicity in general by pulling the $G$-bundle
over $Y/G$ back to a $G$-bundle over $Y$, which is trivial;
then $H^{>0}(G,M)\otimes_{O(Y/G)}O(Y)$ is 0
by \cite[Proposition 4.13]{Jantzen}, and so $H^{>0}(G,M)=0$
by faithfully flat descent.

We conclude that the complex computing $H^*(BG,\Omega^j)$
is the same one that computes $H^*(G,S^j(l_G)[-j])$.
\end{proof}

\section{Good filtrations}

In this section, we explain how known results in representation
theory imply calculations of the Hodge cohomology of classifying
spaces in many cases, via Theorem \ref{iso}. This is not logically necessary
for the rest of the paper: Theorem \ref{non-torsion}
is a stronger calculation of Hodge cohomology,
based on ideas from homotopy theory.

Let $G$ be a split reductive group over a field $k$. (A textbook
reference on split reductive groups is \cite[Chapter 21]{Milne}.)
A {\it Schur module }for $G$
is a module of the form $H^0(\lambda)$ for a dominant weight $\lambda$.
By definition, $H^0(\lambda)$ means $H^0(G/B,L(\lambda))$, where
$B$ is a Borel subgroup and $L(\lambda)$ is the line bundle
associated to $\lambda$. For $k$ of characteristic zero,
the Schur modules are exactly the irreducible representations
of $G$. Kempf showed that the dimension of the Schur modules
is independent of the characteristic of $k$ \cite[Chapter II.4]{Jantzen}.
They need not be irreducible in characteristic $p$, however.

A $G$-module $M$ has a {\it good filtration }if there is a sequence
of submodules $0\subset M_0\subset M_1\subset \cdots$ such that
$M=\cup M_j$ and each quotient $M_i/M_{i-1}$ is a Schur module.
One good feature of Schur modules is that their cohomology
groups are known, by Cline-Parshall-Scott-van der Kallen
\cite[Proposition 4.13]{Jantzen}. Namely,
$$H^i(G,H^0(\lambda))\cong \begin{cases}
k &\text{if }i=0 \text{ and }\lambda=0\\
0 &\text{otherwise}.
\end{cases}$$
As a result, $H^i(G,M)=0$ for all $i>0$ when $M$ has a good filtration.

The following result was proved by Andersen-Jantzen and Donkin
\cite[Proposition and proof of Theorem 2.2]{Donkin},
\cite[II.4.22]{Jantzen}. The statement on the ring of invariants
incorporates earlier work by Kac and Weisfeiler.
Say that a prime number $p$ is {\it good }for a reductive group $G$
if $p\neq 2$ if $G$ has a simple factor not of type $A_n$,
$p\neq 2,3$ if $G$ has a simple factor of exceptional type,
and $p\neq 2,3,5$ if $G$ has an $E_8$ factor.

\begin{theorem}
\label{donkin}
Let $G$ be a split reductive group over a field $k$. Assume either
that $G$ is a simply connected semisimple
group and $\cha(k)$ is good for $G$, or that $G=GL(n)$.
Then the polynomial ring
$O(\g)=S(\g^*)$ has a good filtration as a $G$-module, and the ring
of invariants $O(\g)^G$ is a polynomial ring over $k$, with generators
in the fundamental degrees of $G$.
\end{theorem}

It follows that, under these assumptions, $H^{>0}(G,S^j(\g^*))$ is zero
for all $j\geq 0$. Equivalently, $H^i(BG,\Omega^j)=0$ for $i\neq j$,
by Theorem \ref{iso}. We prove this under the weaker
assumption that $p$ is not a torsion prime in Theorem \ref{non-torsion}.

\section{K\"unneth formula}

The K\"unneth formula holds for Hodge cohomology, in the following form.
The hypotheses apply to the main case studied in this paper:
classifying stacks $BG$ with $G$ an affine group scheme of finite
type over a field.

\begin{proposition}
\label{kunneth}
Let $X$ and $Y$ be quasi-compact
algebraic stacks with affine diagonal over a field $k$.
Then
$$H^*_{\Ho}((X\times_k Y)/k)\cong H^*_{\Ho}(X/k)\otimes_k H^*_{\Ho}(Y/k).$$
\end{proposition}

\begin{proof}
Since $X$ and $Y$ are quasi-compact, there are affine schemes
$A$ and $B$ with smooth surjective morphisms $A\arrow X$
and $B\arrow Y$
\cite[Tag 04YA]{Stacks}. Since $X$ and $Y$ have affine diagonal,
the fiber products $A^{n+1}_X$ and $B^{n+1}_Y$ are affine
over the products $A^{n+1}$ and $B^{n+1}$ over $k$,
and so they are affine schemes, for all $n\geq 0$.

The morphism
$A\times B\arrow X\times Y$ is smooth and surjective.
Therefore, the Hodge cohomology of $X\times Y$ is the cohomology
of the \u{C}ech simplicial space $\Cech(A\times B/X\times Y)$ over $k$,
with coefficients in $\Omega^*$ (with zero differential).
This space is the product
$\Cech(A/X)\times \Cech(B/Y)$ over $k$. By the previous
paragraph, these are in fact simplicial affine schemes over $k$.

The quasi-coherent sheaf $\Omega^1$ on the product of two
affine schemes over $k$ is the direct sum of the pullbacks
of $\Omega^1$ from the two factors. (No smoothness is needed
for this calculation.) Therefore,
the quasi-coherent sheaf $\Omega^*$ on the product affine scheme
$A^{n+1}_X\times B^{n+1}_Y$ over $k$ is the tensor product of the pullbacks
on $\Omega^*$ on those two schemes. So
$H^0(A^{n+1}_X\times B^{n+1}_Y,\Omega^*)$ is the tensor product
of $H^0(A^{n+1}_X,\Omega^*)$ and $H^0(B^{n+1}_Y,\Omega^*)$ over $k$.

The spectral sequence of the simplicial scheme
$\Cech(A/X)\times \Cech(B/Y)$ with coefficients
in $\Omega^*$ reduces to one row, since all the schemes
here are affine. Explicitly, by the previous paragraph,
the cohomology of the product simplicial scheme is the
cohomology of the tensor product over $k$ of the two cosimplicial
vector spaces $H^0(A^{n+1}_X,\Omega^*)$ and $H^0(B^{n+1}_Y,\Omega^*)$.
By the Eilenberg-Zilber theorem, it follows that
the cohomology of the product simplicial scheme is the tensor
product over $k$ of the cohomology of the two factors.
\cite[Theorem 29.3]{May}. Equivalently,
$$H^*_{\Ho}((X\times_k Y)/k)\cong H^*_{\Ho}(X/k)\otimes_k H^*_{\Ho}(Y/k).$$
\end{proof}

\section{Parabolic subgroups}

\begin{theorem}
\label{parabolic}
Let $P$ be a parabolic subgroup of a reductive group $G$ over a field $k$,
and let $L$ be the Levi quotient of $P$ (the quotient of $P$
by its unipotent radical). Then the restriction
$$H^i(BP,\Omega^j)\arrow H^i(BL,\Omega^j)$$
is an isomorphism for all $i$ and $j$. Equivalently,
$$H^a(P,S^j(\p^*))\arrow H^a(L,S^j(\l^*))$$
is an isomorphism for all $a$ and $j$.
\end{theorem}

Theorem \ref{parabolic} can be viewed as a type of homotopy invariance
for Hodge cohomology of classifying spaces. This is not automatic,
since Hodge cohomology is not $A^1$-homotopy invariant
for smooth varieties. Homotopy invariance of Hodge cohomology
also fails in general
for classifying spaces. For example, let $G_a$ be the additive
group over a field $k$. Then the Hodge cohomology group
$H^1(BG_a,O)$ is not zero for any $k$, and it is a $k$-vector
space of infinite dimension for $k$ of positive characteristic;
this follows from Theorem \ref{additive}, due to Cline, Parshall,
Scott, and van der Kallen, together with Theorem \ref{iso}.

\begin{proof}
(Theorem \ref{parabolic})
Let $U$ be the unipotent radical of $P$, so that $L=P/U$.
It suffices to show that 
$$H^a(P,S^j(\p^*))\arrow H^a(L,S^j(\l^*))$$
is an isomorphism after extending the field $k$. So we can assume
that $G$ has a Borel subgroup $B$ and that $B$ is contained in $P$.
Let $R$ be the set of roots for $G$.
We follow the convention that the weights of $B$ acting on the Lie
algebra of its unipotent radical are the {\it negative }roots $R^{-}$.
There is a subset $I$ of the set $S$ of simple roots so that
$P$ is the associated subgroup $P_I$, in the notation of
\cite[II.1.8]{Jantzen}.
More explicitly, let $R_I=R\cap\Z I$; then
$P=P_I$ is the semidirect product $U_I\rtimes L_I$,
where $L_I$ is the reductive group $G(R_I)$ and $U:=U_I$
is the unipotent group $U((-R^+)\setminus R_I)$.

As a result, the weights of $P$ on $\p$ are all the roots $\sum_{\alpha\in S}
n_{\alpha}\alpha$ such that $n_{\alpha}\leq 0$ for $\alpha$ not in $I$.
The coefficients $n_{\alpha}$ for $\alpha$ not in $I$
are all zero exactly for the weights of $P$ on $\p/\uu$.
As a result, for any $j\geq 0$, the weights of $P$ on $S^j(\p^*)$
are all in the root lattice, with nonnegative coefficients
for the simple roots not in $I$, and with those coefficients all zero
only for the weights of $P$ on the subspace $S^j((\p/\uu)^*)
\subset S^j(\p^*)$. 

We now use the following information about the cohomology of $P$-modules
\cite[Proposition II.4.10]{Jantzen}. For any element $\lambda$
of the root lattice $\Z S$, $\lambda=\sum_{\alpha\in S}n_{\alpha}\alpha$,
the {\it height }$\height(\lambda)$ means the integer $\sum_{\alpha\in S}
n_{\alpha}$.

\begin{proposition}
\label{vanishing}
Let $P$ be a parabolic subgroup of a reductive group $G$ over a field,
and let $M$ be a $P$-module. If $H^j(P,M)\neq 0$ for some $j\geq 0$,
then there is a weight $\lambda$ of $M$ with $-\lambda\in \N R^+$
and $\height(\lambda)\geq j$.
\end{proposition}

Given the information above about the weights of $P$ on $S^j(\p^*)$,
it follows that the homomorphism
$$H^a(P,S^j(\p^*))\arrow H^a(P,S^j((\p/\uu)^*))$$
is an isomorphism for all $a$ and $j$. Here $\p/\uu\cong \l$
is a representation of the quotient group $L=P/U$. It remains
to show that the pullback
$$H^a(L,S^j((\p/\uu)^*))\arrow H^a(P,S^j((\p/\uu)^*))$$
is an isomorphism. This would not be true for an arbitrary
representation of $L$; we will have to use what we know about
the weights of $L$ on $S^j((\p/\uu)^*)$.

We also use the following description of the cohomology of
an additive group $V=(G_a)^n$ over a perfect field $k$
\cite[Proposition I.4.27]{Jantzen}. (To prove Theorem \ref{parabolic},
we can enlarge the field $k$, and so we can assume that $k$
is perfect.) The following description
is canonical, with respect to the action of $GL(V)$ on $H^*(V,k)$.
Write $W^{(j)}$ for the $j$th Frobenius twist of a vector space
$W$, as a representation of $GL(W)$.

\begin{theorem}
\label{additive}
(1) If $k$ has characteristic zero, then $H^*(V,k)\cong \Lambda(V^*)$,
with $V^*$ in degree 1.

(2) If $k$ has characteristic 2, then
$$H^*(V,k)\cong S(\oplus_{j\geq 0}(V^*)^{(j)}),$$
with all the spaces $(V^*)^{(j)}$ in degree 1.

(3) If $k$ has characteristic $p>2$, then
$$H^*(V,k)\cong \Lambda(\oplus_{j\geq 0}(V^*)^{(j)})\otimes
S(\oplus_{j\geq 1}(V^*)^{(j)}),$$
with all the spaces $(V^*)^{(j)}$ in the first factor in degree 1,
and all the spaces $(V^*)^{(j)}$ in the second factor in degree 2.
\end{theorem}

We also use the Hochschild-Serre spectral sequence for the cohomology
of algebraic groups \cite[I.6.5, Proposition I.6.6]{Jantzen}:

\begin{theorem}
\label{hsss}
Let $G$ be an affine group scheme of finite type over a field $k$,
and let $N$ be a normal $k$-subgroup scheme of $G$. For every $G$-module
(or complex of $G$-modules) $V$,
there is a spectral sequence
$$E_2^{ij}=H^i(G/N,H^j(N,V))\imp H^{i+j}(G,V).$$
\end{theorem}

Theorems \ref{additive} and \ref{hsss}
give information about the weights of $L$ on $H^*(U,k)$,
that is, about the action of a maximal torus $T\subset L$
on $H^*(U,k)$. The method is to write $U$ (canonically) as an extension
of additive groups $V=(G_a)^n$ and use the Hochschild-Serre
spectral sequence. We deduce that as a representation of $L$,
all weights of $H^{>0}(U,k)$ are in the root lattice of $G$,
with nonnegative coefficients for the simple roots not in $I$,
and with at least one of those coefficients positive. (This is the
same sign as we have for the action of $L$ on $\uu^*$.)

Now apply the Hochschild-Serre spectral sequence to the normal
subgroup $U$ in $P$:
$$E_2^{ij}(L,H^j(U,k)\otimes S^l((\p/\uu)^*))\imp H^{i+j}(P,S^l((\p/\uu)^*)).$$
By the analysis of $S^l(\p^*)$ above, all the weights of $L$ on
the subspace $S^l((\p/\uu)^*)$ are in the root lattice of $G$, and
the coefficients of all simple roots not in $I$ are equal to zero.
Combining this with the previous paragraph, we find:
for $l\geq 0$ and $j>0$, all weights of $L$ on
$H^j(U,k)\otimes S^l((\p/\uu)^*)$ have all coefficients of the simple
roots not in $I$ nonnegative, with at least one positive.
By Proposition \ref{vanishing}, it follows that
$$H^i(L,H^j(U,k)\otimes S^l((\p/\uu)^*))=0$$
for all $i$ and $l$ and all $j>0$. So the spectral sequence
above reduces to an isomorphism
$$H^i(P,S^l((\p/\uu)^*))\cong H^i(L,S^l((\p/\uu)^*)),$$
as we wanted. Theorem \ref{parabolic} is proved.
\end{proof}

\section{Pushforward on Hodge cohomology}

Gros constructed a cycle map $CH^i(X)\arrow H^i(X,\Omega^i)$
for smooth schemes over a perfect field \cite{Gros}.
He also showed that the cycle map is compatible with proper pushforward,
in the following sense \cite[sections II.2 and II.4]{Gros}

\begin{proposition}
\label{pushforward}
Let $f\colon X\arrow Y$ be a proper morphism of smooth
schemes over a field $k$, and assume
that $\dim(X)-\dim(Y)=N$ everywhere. Then there is a pushforward
homomorphism
$$f_*\colon H^i(X,\Omega^i)\arrow H^{i-N,j-N}(Y,\Omega^{j-N}).$$
This is compatible with the cycle map, via a commutative diagram:
\cite[section II.4]{Gros}:
$$\xymatrix@C-10pt@R-10pt{
CH^i(X)\ar[r]\ar[d] & H^i(X,\Omega^i)\ar[d]\\
CH^i(Y)\ar[r] & H^i(Y,\Omega^i)
}$$
\end{proposition}

\section{Hodge cohomology of flag manifolds}

\begin{proposition}
\label{flag}
Let $P$ be a parabolic subgroup
of a split reductive group $G$ over a field $k$.
Then the cycle map
$$CH^*(G/P)\otimes_{\Z}k \arrow H^*_{\Ho}((G/P)/k)$$
is an isomorphism of $k$-algebras.
In particular, $H^i(G/P,\Omega^j)=0$ for $i\neq j$.
\end{proposition}

This is well known for $k$ of characteristic zero, but the general
result is also not difficult.
Andersen gave the additive
calculation of $H^i(G/P,\Omega^j)$ in any characteristic
\cite[Proposition II.6.18]{Jantzen}. Note that Chevalley and Demazure
gave combinatorial descriptions of the Chow ring of $G/P$, which
in particular show that this ring is independent of $k$, and isomorphic
to the ordinary cohomology ring $H^*(G_{\C}/P_{\C},\Z)$
\cite[Proposition 11]{Chevalley}, \cite{Demazure}.
(That makes sense because the classification
of split reductive groups and their parabolic subgroups
is the same over all fields.)

\begin{proof}
(Proposition \ref{flag})
We use that $X=G/P$ has a cell decomposition, the Bruhat decomposition.
It follows that the Chow group of $X$ is free abelian on the set
of cells. In fact, the Chow motive of $X$ is isomorphic to a direct
sum of Tate motives $\Z(a)$, indexed by the cells \cite[2.6]{Scholl}.

Next, Hodge cohomology is a functor on Chow motives over $k$.
(That is, we have to show
that a correspondence between smooth projective varieties
gives a homomorphism on Hodge cohomology, which
follows from Gros's cycle map and proper pushforward for Hodge
cohomology (Proposition \ref{pushforward}).)
As a result, the calculation follows from the Hodge
cohomology of projective space, which implies that the Chow motive
$M=\Z(a)$ has Hodge cohomology $H^i(M,\Omega^j)$ isomorphic to $k$
if $i=j=a$ and zero otherwise.
\end{proof}

\section{Invariant functions on the Lie algebra}

\begin{theorem}
\label{invariant}
Let $G$ be a reductive group over a field $k$, $T$ a maximal torus
in $G$, $\g$ and $\t$ the Lie algebras. If $k$ has characteristic $p>0$,
assume that no root of $G$ is divisible by $p$ in the weight
lattice $\Hom(T,G_m)$.
Then the restriction $O(\g)^G\arrow O(\t)^W$ is an isomorphism.
\end{theorem}

Theorem \ref{invariant} was proved by Springer and Steinberg
for any adjoint group $G$,
in which case the assumption on the roots always holds
\cite[II.3.17']{SS}.
If we do not assume that $G$ is adjoint, then the assumption
on the roots is necessary, as shown by the example of the symplectic
group $\Sp(2n)$ in characteristic 2 (where some roots are divisible by 2
in the weight lattice, and the conclusion fails,
as discussed in the proof of Theorem \ref{integral});
but that is the only exception among simple groups.

In particular, Theorem \ref{invariant} applies to cases such as
the spin group $\Spin(n)$ in characteristic 2 with $n\geq 6$,
which we study further in Theorem \ref{spin}.

\begin{proof}
It suffices to prove that the map is an isomorphism after enlarging $k$
to be algebraically closed.
Define a morphism $\varphi\colon G/T\times \t\arrow \g$
by $(gT,x)\mapsto gxg^{-1}\in \g$. Let the Weyl group $W=N_G(T)/T$
act on $G/T\times \t$ by $w(gT,x)=(gw^{-1}T,wxw^{-1})$; then
$\varphi$ factors through the quotient variety $W\backslash (G/T\times \t)$.
Since we assume that no root of $G$ is divisible by $p=\cha(k)$,
each root of $G$ determines a nonzero linear map $\t\arrow k$.
So there
is a {\it regular }element $x$ of $\t$, meaning an element
on which all roots are nonzero.

It follows that the derivative of $\varphi$ at $(1\cdot T, x)$
is bijective. (Indeed, the image of the derivative is at this point
is $\t$ plus the image of $\ad(x):\g/\t\arrow \g$. The vector 
space $\g/t$ is a direct sum of the 1-dimensional root spaces
as a representation of $T$,
and $x$ acts by a nonzero scalar on each space since $x$ is regular.)
So $\varphi\colon G/T\times \t\arrow \g$ is a separable dominant
map. 

Next, I claim that $W\backslash (G/T\times \t)\arrow \g$ is generically
bijective; then it will follow that this map is birational. Note
that the vector space $\t$ is defined over $\F_p$, by the isomorphism
$\t\cong\Hom(G_m,T)\otimes_{\Z}k$ (or over $\Z$, if $k$ has characteristic 0).
Let $x$ be a regular element of $\t$ which is not in any hyperplane defined
over $\F_p$ (or over $\Z$, if $k$ has characteristic 0).
Then our claim follows if the inverse image of $x$
in $G/T\times \t$ is only the $W$-orbit of $(1\cdot T,x)$. 
Equivalently, we have to show that any element $g$ of $G(k)$ that conjugates
$x$ into $\t$ lies in the normalizer $N_G(T)$. 

First suppose that $p>0$. Then,
for any $g\in G(k)$, the intersection of $T$ with $gTg^{-1}$
has $p$-torsion subgroup scheme $H$ contained in $T[p]\cong (\mu_p)^l$,
where $l$ is the dimension of $T$. Here the Lie algebra of $T[p]$ is equal
to the Lie algebra of $T$, and the Lie algebra of $H$ is defined over $\F_p$
in terms of the $\F_p$-structure above on $\t$. So if $gxg^{-1}$ is in $\t$,
then $g \t g^{-1}=\t$, since $x$ is contained in no hyperplane of $\t$
defined over $\F_p$. For $p=0$, the same conclusion holds,
since the Lie algebra of $T\cap gTg^{-1}$ is a subspace of $\t$
defined over $\Z$. The rest of the argument works for any $p\geq 0$.
Let $E$ be the normalizer of $\t$ in $G$; then we have shown
that $g\in E(k)$.

Clearly $E$ contains $T$. Also, the Lie algebra of $E$ is
$\{y\in\g: [y,\t]\subset \t\}$. Since $\t$ acts nontrivially on each
of the 1-dimensional root spaces which span $\g/\t$, the Lie algebra
of $E$ is equal to $\t$. Thus $E$ is smooth over $k$, with identity
component equal to $T$. So $E$ is contained in $N_G(T)$. The reverse
inclusion is clear, and so $E=N_G(T)$. Thus the element $g$ above
is in $N_G(T)$, proving our claim.

As mentioned above, it follows that the morphism $\alpha\colon
W\backslash (G/T\times \t)
\arrow \g$ is birational. This map is also $G$-equivariant, where
$G$ acts on $G/T$ and by conjugation on $\g$.
Because $\alpha$ is dominant, the restriction
$O(\g)^G\arrow O(\t)^W$ is injective. Because $\alpha$ is birational,
every $W$-invariant polynomial $f$ on $\t$ corresponds to a $G$-invariant
rational function on $\g$. We follow Springer-Steinberg's
argument: write $f=f_1/f_2$ with $f_1$ and $f_2$ relatively
prime polynomials. The center $Z(G)$ acts trivially on $\g$.
Since $G/Z(G)$ equals its own commutator subgroup,
every homomorphism $G/Z(G)\arrow G_m$ is trivial, and so both
$f_1$ and $f_2$ are $G/Z(G)$-invariant. Thus $O(\t)^W$ is contained
in the fraction field of $O(\g)^G$. Since the ring $O(\g)^G$ is normal,
it follows that $O(\t)^W=O(\g)^G$, as we want.
\end{proof}

\section{Hodge cohomology of $BG$ at non-torsion primes}

\begin{theorem}
\label{non-torsion}
Let $G$ be a reductive group over a field $k$ of characteristic $p\geq 0$.
Then $H^{>0}(G,O(\g))=0$ if and only if $p$ is not a torsion prime
for $G$.
\end{theorem}

\begin{theorem}
\label{integral}
Let $G$ be a split reductive group over $\Z$, and let $p$ be a non-torsion
prime for $G$. Then $H^j(BG_{\Z},\Omega^i)$ localized at $p$ is zero
for $i\neq j$. Moreover, the Hodge cohomology ring $H^*(BG_{\Z},\Omega^*)$
and the de Rham cohomology $H_{\dR}^*(BG/{\Z})$, localized at $p$,
are polynomial rings on generators of degrees equal
to 2 times the fundamental degrees of $G$.
These rings are isomorphic to the cohomology of the topological
space $BG_{\C}$ with $\Z_{(p)}$ coefficients.
\end{theorem}

We recall the definition of torsion primes for a reductive group $G$
over a field $k$. Let $B$ be a Borel subgroup of $G_{\overline{k}}$,
and $T$ a maximal torus in $B$.
Then there is a natural homomorphism from the character group
$X^*(T)=\Hom(T,G_m)$ (the weight lattice of $G$) to the Chow group
$CH^1(G_{\overline{k}}/B)$.
Therefore, for $N=\dim(G_{\overline{k}}/B)$,
there is a homomorphism from the symmetric
power $S^N(X^*(T))$ to $CH^N(G_{\overline{k}}/B)$; taking the degree
of a zero-cycle on $G_{\overline{k}}/B$ gives a homomorphism
(in fact, an isomorphism) $CH^N(G_{\overline{k}}/B)\arrow\Z$.
A prime number $p$ is said
to be a torsion prime for $G$ if the image of $S^N(X^*(T))\arrow \Z$
is zero modulo $p$. Borel showed that
$p$ is a torsion prime for $G$ if and only if
the cohomology $H^*(BG_{\C},\Z)$ has $p$-torsion, where $G_{\C}$ is the
corresponding complex reductive group.
Various other characterizations of the torsion primes
for $G$ are summarized in \cite[section 1]{Totarospin}.

In most cases, Theorem \ref{non-torsion}
follows from Theorem \ref{donkin}. Explicitly,
a prime number $p$ is non-torsion for a simply connected simple group $G$
if $p\neq 2$ if $G$ has a simple factor not of type $A_n$ or $C_n$,
$p\neq 2,3$ if $G$ has a simple factor of type $F_4$, $E_6$,
$E_7$, or $E_8$,
and $p\neq 2,3,5$ if $G$ has an $E_8$ factor.
So the main new
cases in Theorem \ref{non-torsion} are
the symplectic groups $\Sp(2n)$ in characteristic 2 and $G_2$
in characteristic 3. (These are non-torsion primes, but not good
primes in the sense of Theorem \ref{donkin}.)
In these cases, the representation-theoretic
result that $H^{>0}(G,O(\g))=0$ seems to be new. Does
$O(\g)$ have a good filtration in these cases?

The following spectral sequence, modeled on the Leray-Serre
spectral sequence in topology, will be important for the rest
of the paper.

\begin{proposition}
\label{ls}
Let $P$ be a parabolic subgroup of a split reductive group $G$ over
a field $k$. Let $L$ be the quotient of $P$ by its unipotent radical.
Then there is a spectral sequence of algebras
$$E_2^{ij}=H^i_{\Ho}(BG/k)\otimes H^j_{\Ho}((G/P)/k)
\imp H^{i+j}_{\Ho}(BL/k).$$
\end{proposition}

\begin{proof}
Consider $\Omega^*=\oplus \Omega^i$ as a presheaf of commutative dgas
on smooth $k$-schemes, with zero differential.

For a smooth morphism $f\colon X\arrow Y$ of smooth $k$-schemes,
consider the object $Rf_*(\Omega^*_X)$ in the derived category $D(Y)$
of etale sheaves on $Y$.
Here the sheaf $\Omega^*_X$ on $X$ has an increasing filtration, compatible
with its ring structure, with
0th step the subsheaf $f^*(\Omega^*_Y)$ and $j$th graded piece
$f^*(\Omega^*_Y)\otimes \Omega^j_{X/Y}$. So $Rf_*(\Omega^*_X)$
has a corresponding filtration in $D(Y)$, with $j$th graded
piece $Rf_*(f^*(\Omega^*_Y)\otimes \Omega^j_{X/Y})\cong
\Omega^*_Y\otimes Rf_*\Omega^j_{X/Y}$. This gives a spectral sequence
$$E_2^{ij}=H^{i+j}(Y,\Omega^*_Y\otimes Rf_*\Omega^j_{X/Y})
\imp H^{i+j}(X,\Omega^*_X).$$

Now specialize to the case where $f\colon X\arrow Y$ is the $G/P$-bundle
associated to a principal $G$-bundle over $Y$.
The Hodge cohomology of $G/P$ is essentially independent of the base field,
by the isomorphism $H^*_{\Ho}((G/P)/k)\cong CH^*(G/P)\otimes_{\Z}k$
(Proposition \ref{flag}). Therefore, each object $Rf_*(\Omega^j_{X/Y})$
is a trivial vector bundle on $EG/P$, with fiber $H^j(G/P,\Omega^j)$, viewed
as a complex in degree $j$.
So we can rewrite the spectral sequence as
$$E_2^{ij}=H^{i}(Y,\Omega^*)\otimes H^j(G/P,\Omega^j)
\imp H^{i+j}(X,\Omega^*).$$

All differentials in the spectral sequence above
preserve the degree in the grading of $\Omega^*$. Therefore,
we can renumber the spectral sequence so that it is graded by total
degree:
$$E_2^{ij}=H^{i}_{\Ho}(Y/k)\otimes H^j_{\Ho}((G/P)/k)
\imp H^{i+j}_{\Ho}(X/k).$$

Finally, we consider the analogous spectral sequence
for the morphism $f\colon EG/P\arrow BG$ of simplicial schemes:
$$E_2^{ij}=H^i_{\Ho}(BG/k)\otimes H^j_{\Ho}((G/P)/k)
\imp H^{i+j}_{\Ho}((EG/P)/k).$$
By Lemma \ref{EG},
the output of the spectral sequence is isomorphic to $H^*_{\Ho}(BP/k)$,
or equivalently (by Theorem \ref{parabolic}) to $H^*_{\Ho}(BL/k)$.
This is a spectral sequence of algebras. All differentials preserve
the degree in the grading of $\Omega^*$.
\end{proof}

\begin{proof}
(Theorem \ref{non-torsion})
First, suppose that $H^{>0}(G,O(\g))=0$;
then we want to show that $\cha(k)$ is not a torsion prime for $G$.
By Theorem \ref{iso}, the assumption implies that
$H^j(BG,\Omega^i)=0$ for all $i\neq j$.
Apply Proposition \ref{ls} when $P$ is a Borel subgroup $B$ in $G$;
this gives a spectral sequence
$$E_2^{ij}=H^i_{\Ho}(BG/k)\otimes H^j_{\Ho}((G/B)/k)
\imp H^{i+j}_{\Ho}(BT/k),$$
where $T$ is a maximal torus in $B$.
Under our assumption, this spectral sequence degenerates at $E_2$,
because the differential $d_r$ (for $r\geq 2$)
takes $H^i(BG,\Omega^i)\otimes H^j(G/B,\Omega^j)$ 
into $H^{i+r}(BG,\Omega^{i+r-1})\otimes H^{j-r+1}(G/B,
\Omega^{j-r+1})$, which is zero. It follows that $H^*_{\Ho}(BT/k)
\arrow H^*_{\Ho}((G/B)/k)$ is surjective. Here $H^*_{\Ho}(BT/k)$
is the polynomial ring $S(X^*(T)\otimes k)$
by Theorem \ref{donkin}, and $H^*_{\Ho}((G/B)/k)=
CH^*(G/B)\otimes k$ by Proposition \ref{flag}. It follows
that the ring $CH^*(G/B)\otimes k$ is generated as a $k$-algebra
by the image of $X^*(T)\arrow CH^1(G/B)$.
Equivalently, $p$ is not a torsion prime for $G$.

Conversely, suppose that $p$ is not a torsion prime for $G$.
That is, the homomorphism $S(X^*(T)\otimes k)\arrow
CH^*(G/B)\otimes k$ is surjective. Equivalently,
$H^*_{\Ho}(BT/k)\arrow H^*_{\Ho}((G/B)/k)$ is surjective.
By the product structure on the spectral sequence above,
it follows that
the spectral sequence degenerates at $E_2$.
Since $H^j(BT,\Omega^i)=0$ for $i\neq j$, it follows that
$H^j(BG,\Omega^i)=0$ for $i\neq j$. Equivalently,
$H^{>0}(G,O(\g))=0$.
\end{proof}

\begin{proof}
(Theorem \ref{integral})
Let $G$ be a split reductive group over $\Z$, and let $p$ be a non-torsion
prime for $G$. We have a short exact sequence
$$0\arrow H^j(BG_{\Z},\Omega^i)/p\arrow H^j(BG_{\F_p},\Omega^i)
\arrow H^{j+1}(BG_{\Z},\Omega^i)[p]\arrow 0.$$
By Theorem \ref{non-torsion},
the Hodge cohomology ring $H^*(BG_{\Z},\Omega^*)$ localized at $p$
is concentrated
in bidegrees $H^{i,i}$ and is torsion-free. This ring tensored
with $\Q$ is the ring of invariants $O({\g}_{\Q})^G$, which
is a polynomial ring on generators of degrees equal
to the fundamental degrees of $G$.

To show that the Hodge cohomology ring over $\Z_{(p)}$
is a polynomial ring on generators in $H^{i.i}$ for $i$
running through the fundamental degrees of $G$,
it suffices to show that the Hodge cohomology ring
$H^*_{\Ho}(BG/\F_p)$ is a polynomial ring in the
same degrees. Given that, the other statements
of the theorem will follow. Indeed, the statement
on Hodge cohomology implies that the de Rham cohomology ring
$H^*_{\dR}(BG/\Z)$ localized at $p$ is also a polynomial
ring, on generators in 2 times the fundamental degrees of $G$.
The cohomology of the topological space $BG_{\C}$ localized
at $p$ is known to be a polynomial ring on generators
in the same degrees, by Borel
\cite[section 1]{Totarospin}.

From here on, let $k=\F_p$, and write $G$ for $G_k$.
By definition of the Weyl group $W$ as $W=N_G(T)/T$, the image
of $H^*_{\Ho}(BG/k)$ in $H^*_{\Ho}(BT/k)=S(X(T)\otimes k)$ is contained
in the subring of $W$-invariants. We now use that $p$ is not a torsion
prime for $G$. By Demazure, except in the case
where $p=2$ and $G$ has an $\Sp(2n)$ factor, the ring
of $W$-invariants in $S(X(T)\otimes k)$ is a polynomial algebra over $k$,
with the degrees of generators equal to the fundamental degrees
of $G$ \cite[Th\'eor\`eme]{Demazure}. 

By Theorem \ref{invariant}, for any simple group $G$
over a field $k$ of characteristic $p$ with $p$ not a torsion prime, except
for $G=\Sp(2n)$ with $p=2$, the restriction $O(\g)^G\arrow O(\t)^W$
is an isomorphism. In particular,
for $G=SL(n)$ with $n\geq 3$ over any field $k$,
it follows that
$O(\g)^G$ is a polynomial ring
with generators in the fundamental degrees of $G$, that is,
$2,3,\ldots,n$.

The case of $\Sp(2n)$ in characteristic 2
(including $SL(2)=\Sp(2)$) is a genuine
exception: here $O(\g)^G$ is a subring of $O(\t)^W$, not equal to it.
However, it is still true in this case that $O(\g)^G$ is a polynomial
ring with generators in the fundamental degrees of $G$, that is,
$2,4,\ldots,2n$. One way to check this is first to compute that,
for $k$ of characteristic 2,
$O(\sl(2))^{SL(2)}$ is the subring $k[c_2]$ of $O(\t)^W=k[x_1]$,
where $x_1$ is in degree 1, $c_2$ is in degree 2, and $c_2\mapsto x_1^2$.
(Note that $W\cong \Z/2$ acts trivially on $\t\cong k$ since the characteristic
is 2.) Here $c_2$ is the determinant on the space $\sl(2)$ of matrices
of trace zero, and $O(\sl(2))^{SL(2)}$ is only $k[c_2]$ (not $k[x_1]$)
because the determinant 
$$\det \begin{pmatrix} a & b\\ c& a\end{pmatrix}=a^2+bc$$
is visibly not a square in $O(\sl(2))^{SL(2)}$. To handle $G=\Sp(2n)$
for any $n$, note that the inclusion of $O(\sp(2n))^{\Sp(2n)}$
into $O(\t)^W$ factors through $((O(\sl(2))^{SL(2)})^n)^{S_n}$,
because of the subgroup $S_n\ltimes SL(2)$ in $\Sp(2n)$. By the calculation
for $SL(2)$, $((O(\sl(2))^{SL(2)})^n)^{S_n}$ is isomorphic
to $k[c_2,c_4,\ldots,c_{2n}]$ where $c_{2i}$ is in degree $2i$;
so $O(\sp(2n))^{\Sp(2n)}$ is a subring of that polynomial ring.
Conversely, the even coefficients of the characteristic polynomial
for a matrix in $\sp(2n)\subset \gl(2n)$ restrict to these
classes $c_{2i}$, and so $O(\sp(2n))^{\Sp(2n)}$ is isomorphic
to the polynomial ring $k[c_2,c_4,\ldots,c_{2n}]$, as we want.
\end{proof}

\section{$\mu_p$}

\begin{proposition}
\label{mup}
Let $k$ be a field of characteristic $p>0$.
Let $G$ be the group scheme $\mu_p$ of $p$th roots of unity
over $k$. Then 
$$H^*_{\Ho}(B\mu_p/k)\cong k[c_1]\langle v_1\rangle,$$
where $c_1$ is in $H^1(B\mu_p,\Omega^1)$ and $v_1$ is in
$H^0(B\mu_p,\Omega^1)$.
Likewise $H^*_{\dR}(B\mu_p/k)\cong k[c_1]\langle v_1\rangle$
with $|v_1|=1$ and $|c_1|=2$.
\end{proposition}

Here $R\langle v\rangle$ denotes the exterior algebra
over a graded-commutative ring $R$ with generator $v$; that is,
$R\langle v\rangle=R\oplus R\cdot v$, with product $v^2=0$.
See section \ref{notation}
for the definitions of Hodge and de Rham cohomology
we are using for a non-smooth group scheme such as $\mu_p$.
Proposition \ref{mup} can help to compute
Hodge cohomology of $BG$ for smooth group schemes $G$, as we will
see in the proof of Theorem \ref{son} for $G=SO(n)$.

Proposition \ref{mup} is roughly what the topological
analogy would suggest. Indeed, for $k$ of characteristic $p$,
the ring $H^*((B\mu_p)_{\C},k)$ is a polynomial ring $k[x]$
with $|x|=1$ if $p=2$, or a free graded-commutative algebra
$k\langle x,y\rangle$ with $|x|=1$ and $|y|=2$ if $p$ is odd.
So $H^*_{\dR}(B\mu_p/k)$ is isomorphic to $H^*((B\mu_p)_{\C},k)$
additively for any prime $p$, and as a graded ring if $p>2$.

\begin{proof}
Let $G=\mu_p$ over $k$. The co-Lie complex $l_G$ in the derived category
of $G$-modules, discussed in section \ref{coLie},
has $H^0(l_G)\cong \g^*\cong k$ and also
$H^{-1}(l_G)\cong k$, with other cohomology groups being zero.
(In short, this is because $G$ is a complete intersection in the
affine line, defined by the one equation $x^p=1$.)

Since representations of $G$ are completely reducible,
we have $\Ext_G^{>0}(M,N)=0$ for all $G$-modules $M$ and $N$
\cite[Lemma I.4.3]{Jantzen}.
The isomorphism class of $l_G$ is described by an element
of $\Ext^2_G(k,k)$, which is zero. So $l_G\cong k\oplus k[1]$
in the derived category of $G$-modules.

By Theorem \ref{iso}, we have
$$H^i(BG,\Omega^j)\cong H^{i-j}(G,S^j(l_G)).$$
Here
\begin{align*}
S^j(l_G)&\cong \oplus_{m=0}^j S^m(k)\otimes S^{j-m}(k[1])\\
&\cong \oplus_{m=0}^j S^m(k)\otimes \Lambda^{j-m}(k)[j-m],
\end{align*}
which is isomorphic to $k\oplus k[1]$ if $j\geq 1$ and to $k$
if $j=0$. Therefore, $H^i(BG,\Omega^j)$ is isomorphic to $k$
if $0\leq i=j$ or if $0\leq i=j-1$, and is otherwise zero.

Write $c_1$ for the generator of $H^1(BG,\Omega^1)$, which is
pulled back from the Chern class $c_1$ in $H^1(BG_m,\Omega^1)$ via
the inclusion $G\inj G_m$. Write $v_1$ for the generator
of $H^0(BG,\Omega^1)$. We have $v_1^2=0$ because $H^0(BG,\Omega^2)=0$.
By the proof, Theorem \ref{iso} also describes the ring
structure on the Hodge cohomology of $BG$. In particular,
$\oplus_i H^i(BG,\Omega^i)$ is the ring of invariants of $G$
acting on $O(\g)$, which is the polynomial ring $k[c_1]$.
Finally, the description of $S^j(l_G)$ also shows that
$\oplus_i H^i(BG,\Omega^{i+1})$ is the free module over $k[c_1]$
on the generator $v_1$. This completes the proof that
$$H^*_{\Ho}(BG/k)\cong k[c_1]\langle v_1\rangle.$$

Finally, consider the Hodge spectral sequence for $BG$.
The element $v_1$ is a permanent cycle because $H^0(BG,\Omega^2)=0$,
and $c_1$ is a permanent cycle because it is pulled back
from a permanent cycle on $BG_m$. Therefore, the Hodge
spectral sequence degenerates at $E_1$. We have $v_1^2=0$
in de Rham cohomology as in Hodge cohomology, because
$\oplus_i H^0(BG,\Omega^i)$ is a subring of de Rham cohomology,
using degeneration of the Hodge spectral sequence.
Therefore, the de Rham cohomology of $BG$ is isomorphic
to $k[c_1]\langle v_1\rangle$ as a graded ring.
\end{proof}

\begin{lemma}
\label{finite}
Let $G$ be a discrete group, considered as a group scheme over a field $k$.
Then the Hodge cohomology of the algebraic stack $BG$ is the group
cohomology of $G$:
$$H^i(BG,\Omega^j)\cong \begin{cases}
H^i(G,k)&\text{if }j=0\\
0&\text{otherwise.}
\end{cases}$$
It follows that $H^*_{\dR}(BG/k)\cong H^*(G,k)$.
\end{lemma}

\begin{proof}
Since $G$ is smooth over $k$, we can compute the Hodge cohomology
of the stack $BG$ as the etale cohomology of the simplicial scheme $BG$
with coefficients in $\Omega^j$. Since $G$ is discrete, the sheaf
$\Omega^j$ is zero for $j>0$. For $j=0$, the spectral
sequence
$$E_1^{ab}=H^b(G^a,O)\imp H^{a+b}(BG,O)$$
reduces to a single row, since $H^b(G^a,O)=0$ for $b>0$. That is,
$H^*(BG,O)$ is the cohomology of the standard
complex that computes the cohomology of the group $G$
with coefficients in $k$.
\end{proof}

More generally, we have the following ``Hochschild-Serre''
spectral sequence for the Hodge cohomology
of a non-connected group scheme:

\begin{lemma}
\label{pi0}
Let $G$ be an affine group scheme of finite type over a field $k$.
Let $G^0$ be the identity component of $G$, and suppose that the
finite group scheme $G/G^0$ is the $k$-group scheme associated
to a finite group $Q$. Then there is a
spectral sequence
$$E_2^{ij}=H^i(Q,H^j(BG^0,\Omega^a))\imp H^{i+j}(BG,\Omega^a).$$
for any $a\geq 0$.
\end{lemma}

\begin{proof}
By Theorem \ref{iso}, $H^r(BG,\Omega^a)$ is isomorphic
to $H^{r-a}(G,S^a(l_G))$. The lemma then follows
from the Hochschild-Serre spectral sequence for the cohomology
of $G$ as an algebraic group, Theorem \ref{hsss}.
\end{proof}

\section{The orthogonal groups}

\begin{theorem}
\label{son}
Let $G$ be the split group $SO(n)$ (also called $O^{+}(n)$)
over a field $k$ of characteristic 2. Then
the Hodge cohomology ring of $BG$ is a polynomial ring
$k[u_2,u_3,\ldots,u_n]$, where $u_{2a}$ is in $H^a(BG,\Omega^a)$
and $u_{2a+1}$ is in $H^{a+1}(BG,\Omega^a)$. Also, the Hodge spectral
sequence degenerates at $E_1$, and so $H^*_{\dR}(BG/k)$ is also
isomorphic to $k[u_2,u_3,\ldots,u_n]$.

Likewise, the Hodge and de Rham cohomology rings
of $BO(2r)$ are isomorphic to the polynomial ring
$k[u_1,u_2,\ldots,u_{2r}]$. Finally, the Hodge and de Rham
cohomology rings of $BO(2r+1)$ are isomorphic to
$k[v_1,c_1,u_2,\ldots,u_{2r+1}]/(v_1^2)$,
where $v_1$ is in $H^0(BO(2r+1),\Omega^1)$
and $c_1$ is in $H^1(BO(2r+1),\Omega^1)$.
\end{theorem}

Thus the de Rham cohomology ring of $BSO(n)_{\F_2}$ is isomorphic
to the mod 2 cohomology ring of the topological space $BSO(n)_{\C}$
as a graded ring:
$$H^*(BSO(n)_{\C},\F_2)\cong \F_2[w_2,w_3,\ldots,w_n],$$
where the classes $w_i$ are the Stiefel-Whitney classes.
Theorem \ref{son} gives a new analog of the Stiefel-Whitney classes
for quadratic bundles in characteristic 2. (Note that the $k$-group scheme
$O(2r+1)$ is not smooth in characteristic 2.
Indeed, it is isomorphic to $SO(2r+1)\times \mu_2$.)

The proof is inspired by topology. In particular, it involves
some hard work with spectral sequences, related to Borel's
transgression theorem and Zeeman's comparison theorem.
The method should be useful for
other reductive groups.

The formula for the classes $u_i$
of a direct sum of two quadratic bundles
is not the same as for the Stiefel-Whitney classes in topology.
To state this, define a quadratic form $(q,V)$ over a field $k$
to be {\it nondegenerate }if the radical $V^{\perp}$ of the associated
bilinear form is zero, and {\it nonsingular }if $V^{\perp}$
has dimension at most 1 and $q$ is nonzero on any
nonzero element of $V^{\perp}$.
(In characteristic not 2, nonsingular and nondegenerate are the same.)
The orthogonal group is defined as the automorphism group scheme
of a nonsingular quadratic form \cite[section VI.23]{KMRT}. For example,
over a field $k$ of characteristic 2, the quadratic form
$$x_1x_2+x_3x_4+\cdots+x_{2r-1}x_{2r}$$
is nonsingular of even dimension $2r$, while the form
$$x_1x_2+x_3x_4+\cdots+x_{2r-1}x_{2r}+x_{2r+1}^2$$
is nonsingular of odd dimension $2r+1$, with $V^{\perp}$ of dimension 1.
Let $u_0=1$.

\begin{proposition}
\label{whitney}
Let $X$ be a scheme of finite type over a field $k$ of characteristic 2.
Let $E$ and $F$ be vector bundles with nondegenerate quadratic forms
over $X$ (hence of even rank).
Then, for any $a\geq 0$, in either Hodge cohomology or de Rham cohomology,
$$u_{2a}(E\oplus F)=\sum_{j=0}^a u_{2j}(E)u_{2a-2j}(F)$$
and
$$u_{2a+1}(E\oplus F)=\sum_{l=0}^{2a+1} u_{l}(E)u_{2a+1-l}(F).$$
\end{proposition}

Thus the {\it even }$u$-classes of $E\oplus F$ depend only
on the even $u$-classes of $E$ and $F$.
By contrast, Stiefel-Whitney classes in topology satisfy
$$w_m(E\oplus F)=\sum_{l=0}^m w_l(E)w_{m-l}(F)$$
for all $m$ \cite[Theorem III.5.11]{MT}.

Theorem
\ref{spin} gives an example of a reductive group $G$ for which
the de Rham cohomology of $BG_{\F_p}$ and the mod $p$ cohomology
of $BG_{\C}$ are not isomorphic.
It is a challenge to find out how close these rings are,
in other examples.

Via Theorem \ref{iso}, Theorem \ref{son} can be viewed as a calculation
in the representation theory of the algebraic group $G=SO(n)$ for any $n$,
over a field $k$ of characteristic 2. For example,
when $G=SO(3)=PGL(2)$ over $k$ of characteristic 2, we find
(what seems to be new):
$$H^i(G, S^j(\g^*))\cong \begin{cases} k&\text{if }0\leq i\leq j\\
0&\text{otherwise.}
\end{cases}$$

\begin{proof}
(Theorem \ref{son})
We will assume that $k=\F_2$. This implies the theorem
for any field of characteristic 2.

We begin by computing the ring $\oplus_i H^i(BG,\Omega^i)$
for $G=SO(n)$.
By Theorem \ref{iso}, this is equal to the ring of $G$-invariant polynomial
functions on the Lie algebra $\g$ over $k$. By Theorem \ref{invariant},
since no roots of $G$ are divisible by 2 in the weight lattice for $G$,
the restriction $O(\g)^G\arrow O(\t)^W$ is an isomorphism.

Let $r=\lfloor n/2\rfloor$.
For $n=2r+1$, the Weyl group $W$ is the semidirect product
$S_r\ltimes (\Z/2)^r$. There is a basis $e_1,\ldots,e_r$
for $\t$ on which $(\Z/2)^r$ acts by changing the signs,
and so that action is trivial since
$k$ has characteristic 2. The group $S_r$ has its standard
permutation action on $e_1,\ldots,e_r$. Therefore, the ring
of invariants $O(\t)^W$ is the ring of symmetric functions
in $r$ variables. Let $u_2,u_4,\ldots,u_{2r}$ denote the elementary
symmetric functions. By the isomorphisms mentioned, we can view
$u_{2a}$ as an element of $H^a(BSO(2r+1),\Omega^a)$ for $1\leq a\leq r$,
and $\oplus_i H^i(BSO(2r+1),\Omega^i)$ is the polynomial
ring $k[u_2,u_4,\ldots,u_{2r}]$.

For $n=2r$, the Weyl group $W$ of $SO(2r)$ is the semidirect product
$S_r\ltimes (\Z/2)^{r-1}$. Again, the subgroup $(\Z/2)^{r-1}$ acts
trivially on $\t$, and $S_r$ acts by permutations as usual.
So $\oplus_i H^i(BSO(2r),\Omega^i)$ is also the polynomial
ring $k[u_2,u_4,\ldots,u_{2r}]$, with $u_{2a}$ in $H^a(BSO(2r),
\Omega^a)$ for $1\leq a\leq r$.

For the smooth $k$-group $G=O(2r)$, we can also compute the ring
$\oplus_i H^i(BG,\Omega^i)$. By Theorem \ref{iso},
this is the ring of $G$-invariant polynomial functions on the Lie algebra
$\g=\so(2r)$. This is contained in the ring of $SO(2r)$-invariant
functions on $\g$, and I claim that the two rings are equal. It suffices
to show that an $SO(2r)$-invariant function on $\g$ is also invariant
under the normalizer $N$ in $O(2r)$ of a maximal torus $T$ in $SO(2r)$,
since that normalizer meets both connected components of $O(2r)$.
Here $N=S_r\ltimes (\Z/2)^r$, which acts on $\t$ in the obvious way;
in particular, $(\Z/2)^r$ acts trivially on $\t$. Therefore, an
$SO(2r)$-invariant function on $\g$ (corresponding to an $S_r$-invariant
function on $\t$) is also $O(2r)$-invariant. Thus we have
$\oplus_i H^i(BO(2r),\Omega^i)=k[u_2,u_4,\ldots,u_{2r}]$.

For a smooth group scheme $G$ over $R=\Z/4$,
define the {\it Bockstein}
$$\beta\colon H^i(BG_k,\Omega^j)\arrow H^{i+1}(BG_k,\Omega^j)$$
on the Hodge cohomology of $BG_k$ (where $k=\Z/2$)
to be the boundary homomorphism
associated
to the short exact sequence of sheaves
$$0\arrow \Omega^j_k\arrow \Omega^j_R\arrow \Omega^j_k\arrow 0$$
on $BG_R$. The Bockstein on Hodge cohomology
is not defined for group schemes
such as $\mu_2$ which are flat but not smooth over $R=\Z/4$, because
the sequence of sheaves above need not be exact.

Next, define elements $u_1,u_3,\ldots,u_{2r-1}$ of $H^*_{\Ho}(BO(2r)/k)$
as follows. 
First, let $u_1\in H^1(BO(2r),\Omega^0)$
be the pullback of the generator of $H^1(\Z/2,k)=k$
via the surjection $O(2r)\arrow \Z/2$ (Lemma \ref{finite}).
Next, use that the split group $O(2r)$ over $k=\F_2$ lifts
to a smooth group $O(2r)$ over $\Z$. As a result, we have a Bockstein
homomorphism on the Hodge cohomology of $BO(2r)$.
For $0\leq a\leq r-1$, let $u_{2a+1}=\beta u_{2a}+u_1u_{2a}\in H^{a+1}(BO(2r),
\Omega^a)$. This agrees with the previous formula for $u_1$,
if we make the convention that $u_0=1$.
(The definition of $u_{2a+1}$ is suggested by the formula
for odd Stiefel-Whitney classes in topology:
$w_{2a+1}=\beta w_{2a}+w_1w_{2a}$ \cite[Theorem III.5.12]{MT}.)

I claim that the homomorphism
$$k[u_1,u_2]\arrow H^*_{\Ho}(BO(2)/k)$$
is an isomorphism. To see this,
consider the Hochschild-Serre spectral sequence of Lemma \ref{pi0},
$$E_2^{ij}=H^i(\Z/2,H^j(BSO(2)_k,\Omega^*))\imp H^{i+j}(BO(2)_k,\Omega^*).$$
Here $SO(2)$ is isomorphic to $G_m$, and so we know
the Hodge cohomology of $BSO(2)$ by Theorem \ref{donkin}:
$H^*_{\Ho}(BSO(2)/k)\cong k[c_1]$ with $c_1$ in $H^1(BSO(2),\Omega^1)$.
We read off
that the $E_2$ page of the spectral sequence is the polynomial
ring $k[u_1,u_2]$, with $u_1$ in $H^1(\Z/2,H^0(BSO(2),\Omega^0))$
and $u_2$ in $H^0(\Z/2,H^1(BSO(2),\Omega^1))$. Here $u_1$ is
a permanent cycle, because all differentials send $u_1$
to zero groups. Also, because the surjection $O(2)\arrow \Z/2$
of $k$-groups
is split, there are no differentials into the bottom row of the spectral
sequence; so $u_2$ is also a permanent cycle. It follows that
the spectral sequence degenerates at $E_2$, and hence
that $H^*_{\Ho}(BO(2)/k)\cong k[u_1,u_2]$.

We also need to compute the Bockstein on the Hodge cohomology
of $BO(2)$, which is defined because $O(2)$ lifts to a smooth group scheme
over $R:=\Z/4$. The Bockstein is related to the Hodge cohomology
of $BO(2)_R$ by the exact sequence
$$H^i(BO(2)_R,\Omega^j)\arrow H^i(BO(2)_k,\Omega^j)
\xrightarrow[\beta]{} H^{i+1}(BO(2)_k,\Omega^j).$$
Consider the Hochschild-Serre spectral
sequence of Lemma \ref{pi0} for $BO(2)_R$:
$$E_2^{ij}=H^i(\Z/2,H^j(BSO(2)_R,\Omega^*))\imp H^{i+j}(BO(2)_R,\Omega^*).$$
Here $H^1(BO(2)_R,\Omega^1)$ is isomorphic
to $H^0(\Z/2,H^1(BSO(2)_R,\Omega^1))$, where
$\Z/2$ acts by $-1$ on $H^1(BSO(2)_R,\Omega^1)\cong \Z/4$. So the generator
of $H^1(BO(2)_R,\Omega^1)\cong \Z/2$ maps to zero in $H^1(BO(2)_k,\Omega^1)
=k\cdot u_2$. Therefore, $\beta(u_2)\neq 0$. Since $k=\F_2$, the element
$\beta(u_2)$ in $H^2(BO(2)_k,\Omega^1)=k\cdot u_1u_2$ must be equal
to $u_1u_2$. A similar analysis shows that $\beta(u_1)=u_1^2$.

Finally, think of $O(2)$
as the isometry group of the quadratic form $q(x,y)=xy$ on $V=A^2_k$.
There is an inclusion $H=\Z/2\times \mu_2\subset O(2)$,
where $\Z/2$ switches $x$ and $y$ and $\mu_2$ acts by scalars on $V$.
For later use, it is convenient to say something about
the restriction from $BO(2)$ to $BH$ on Hodge cohomology.
By Lemma \ref{finite},
the Hodge cohomology of $B(\Z/2)$ over $k$ is the cohomology
of $\Z/2$ as a group, namely the polynomial ring
$k[s]$ with $s\in H^1(B(\Z/2),O)$.
Also, by Proposition \ref{mup}, the Hodge cohomology of $B\mu_2$
is $k[t,v]/(v^2)$ with $t\in H^1(B\mu_2,\Omega^1)$
and $v\in H^0(B\mu_2,\Omega^1)$.
Thus we have a homomorphism from $H^*_{\Ho}(BO(2)/k)=k[u_1,u_2]$
to $H^*_{\Ho}(BH/k)\cong k[s,t,v]/(v^2)$
(by the K\"unneth theorem, Proposition \ref{kunneth}).
Here $u_1$ restricts to $s$, since both elements are pulled
back from the generator of $H^1(B\Z/2,O)$. Also,
$u_2$ restricts to either $t$ or $t+sv$, because
$u_2$ restricts to the generator $c_1$ of $H^1(BG_m,\Omega^1)$
and hence to $t$ in $H^1(B\mu_2,\Omega^1)$. Thus
the homomorphism from $H^*_{\Ho}(BO(2)/k)$ to $H^*_{\Ho}(BH/k)/\rad
=k[s,t]$ is an isomorphism. (Here the {\it radical }of a commutative ring
means the ideal of nilpotent elements.)
A direct cocycle computation shows that
$u_2$ restricts to $t+sv$ in $H^1(BH,\Omega^1)$, but we do not
need that fact in this paper.

We now return to the group $O(2r)$ over $k=\F_2$ for any $r$.
I claim that the homomorphism
$$k[u_1,u_2,\ldots,u_{2r}]\arrow H^*_{\Ho}(BO(2r)/k)$$
is injective. The idea is to compose this homomorphism
with restriction to the Hodge cohomology of $BO(2)^r$.
Let $s_1,\ldots,s_r\in H^1(BO(2)^r,\Omega^0)$ be the pullbacks
of $u_1$ from the $r$ $BO(2)$ factors, and let $t_1,\ldots,t_r$
be the pullbacks of $u_2$ from those $r$ factors. By the K\"unneth
theorem (Proposition \ref{kunneth}), the Hodge cohomology
of $BO(2)^r$ is the polynomial ring $k[s_1,\ldots,s_r,t_1,\ldots,t_r]$.
The elements $u_2,u_4,\ldots,u_{2r}$ restrict to the elementary
symmetric functions in $t_1,\ldots,t_r$:
$$u_{2a}\mapsto e_a(t_1,\ldots,t_r)=\sum_{1\leq i_1<\cdots<i_a\leq r}
t_{i_1}\cdots t_{i_a}.$$
Also,
$$u_1\mapsto s_1+\cdots+s_r.$$

The inclusion $O(2)^2\subset O(2r)$ lifts to an inclusion
of smooth groups over $\Z$, and so the restriction homomorphism
commutes with the Bockstein. Therefore, for $0\leq a\leq r-1$,
\begin{align*}
u_{2a+1}&=\beta u_{2a}+u_1u_{2a}\\
&\mapsto \beta\bigg( \sum_{1\leq i_1<\cdots<i_a\leq r}
t_{i_1}\cdots t_{i_a}\bigg)
+(s_1+\cdots+s_{r})\bigg( \sum_{1\leq i_1<\cdots<i_a\leq r}
t_{i_1}\cdots t_{i_a}\bigg) \\
&=\sum_{1\leq i_1<\cdots<i_a\leq r}\bigg( \sum_{j=1}^a
s_{i_j}+\sum_{m=1}^{r}s_m \bigg) t_{i_1}\cdots t_{i_a}\\
&=\sum_{m=1}^r s_m\sum_{\substack{1\leq i_1<
\cdots <i_{a}\leq r\\ \text{none equal to }m}}t_{i_1}
\cdots t_{i_{a}}.
\end{align*}

We want to show that this homomorphism $k[u_1,\ldots,u_{2r}]
\arrow k[s_1,\ldots,s_r,t_1,\ldots,t_r]$ is injective. We can factor
this homomorphism through $k[u_1,u_3,\ldots,u_{2r-1},t_1,\ldots,t_r]$,
by the homomorphism $\rho$
sending $u_2,u_4,\ldots,u_{2r}$ to the elementary symmetric
polynomials in $t_1,\ldots,t_r$. Since $\rho$
is injective, it remains to show that
$$\sigma\colon k[u_1,u_3,\ldots,u_{2r-1},
t_1,\ldots,t_r]\arrow k[s_1,\ldots,s_r,t_1,\ldots,t_r]$$
is injective.

More strongly, we will show that $\sigma$ is generically etale;
that is, its Jacobian determinant is not identically zero.
Because $\sigma$ is the identity on the $t_i$ coordinates,
it suffices to show that the matrix of derivatives
of $u_1,u_3,\ldots,u_{2r-1}$ with respect to $s_1,\ldots,s_r$
is nonzero for $s_1,\ldots,s_r,t_1,\ldots,t_r$ generic.
This matrix of derivatives in fact only involves $t_1,\ldots,t_r$,
because $u_1,u_3,\ldots,u_{2r-1}$ have degree 1 in $s_1,\ldots,s_r$.
For example, for $r=3$, this matrix of derivatives is
$$\begin{pmatrix} 1 & t_2+t_3 & t_2t_3\\
1 & t_1+t_3 & t_1t_3\\
1 & t_1+t_2 & t_1t_2
\end{pmatrix},$$
where the $a$th column gives the derivatives of $u_{2a-1}$
with respect to $s_1,\ldots,s_r$. For any $r$, column 1
consists of 1s, while entry $(j,a)$ for $a\geq 2$
is
$$\sum_{\substack{1\leq i_1<
\cdots <i_{a-1}\leq r\\ \text{none equal to }j}}t_{i_1}
\cdots t_{i_{a-1}}.$$
This determinant is equal to the Vandermonde determinant
$\delta:=\prod_{i<j}(t_i-t_j)$, and in particular
it is not identically zero \cite[Theorem 1]{Dresden}.
(The reference works over $\C$, but it amounts to an identity
of polynomials over $\Z$, which therefore holds over any field.)

Thus we have shown that the composition
$k[u_1,\ldots,u_{2r}]\arrow H^*_{\Ho}(BO(2r)/k)$ is injective,
because the composition to $H^*_{\Ho}(BO(2)^r/k)$ is injective.
Analogously, let us show that $k[u_2,\ldots,u_n]\arrow H^*_{\Ho}(BSO(n)/k)$
is injective for every $n\geq 1$. 

For $n=2r+1$, this is easy, using the inclusions
$O(2)^r\subset O(2r)\subset SO(2r+1)$.
Write $u_2,u_3,\ldots,u_{2r+1}$ for
the elements of the Hodge cohomology of $BSO(2r+1)$ defined by the same
formulas as used above for $BO(2r)$ (which simplify to $u_{2a+1}=\beta u_{2a}$,
since there is no element $u_1$ for $BSO(2r+1)$). Also,
let $v_1,\ldots,v_{2r}$ be the elements of the Hodge cohomology
of $BO(2r)$ that were called $u_1,\ldots,u_{2r}$ above.
Then restricting from $BSO(2r+1)$ to $BO(2r)$ sends
$u_{2a}\mapsto v_{2a}$ and $u_{2a+1}=\beta u_{2a}
\mapsto \beta v_{2a}=v_{2a+1}+v_1v_{2a}$ for $1\leq a\leq r-1$.
It is not immediate how to compute the restriction
of the remaining element $u_{2r+1}$ to $BO(2r)$, but we can compute
its restriction to $BO(2)^r$:
\begin{align*}
u_{2r+1}&=\beta u_{2r}\\
&\mapsto \beta v_{2r}\\
&= \beta(t_1\cdots t_r)\\
&=(s_1+\cdots+s_r)(t_1\cdots t_r).
\end{align*}
Thus, the restriction from $BSO(2r+1)$ to $BO(2)^r$
sends $k[u_2,\ldots,u_{2r+1}]$ into the subring
$$k[v_1,\ldots,v_{2r}]\subset k[s_1,\ldots,s_r,t_1,\ldots,t_r],$$
by $u_{2a}\mapsto v_{2a}$ for $1\leq a\leq r$,
$u_{2a+1}\mapsto v_{2a+1}+v_1v_{2a}$ for $1\leq a\leq r-1$,
and $u_{2r+1}\mapsto v_1v_{2r}$. This homomorphism is injective,
because the corresponding morphism $A^{2r}\arrow A^{2r}$
is birational (for $u_{2r}\neq 0$, one can solve for
$v_1,\ldots,v_{2r}$ in terms of $u_2,\ldots,u_{2r+1}$).
So the homomorphism $k[u_2,\ldots,u_{2r+1}]
\arrow H^*_{\Ho}(BSO(2r+1)/k)$ is injective (because its composition
to $H^*_{\Ho}(BO(2)^r/k)$ is injective).

For $SO(2r)$, we argue a bit differently. As discussed above,
there is a subgroup $\Z/2\times \mu_2\subset O(2)$. Therefore,
we have a $k$-subgroup scheme
$(\Z/2\times \mu_2)^r\subset O(2)^r\subset O(2r)$.
Since $SO(2r)$ is the kernel of a homomorphism from $O(2r)$ onto $\Z/2$,
$SO(2r)$ contains a $k$-subgroup scheme $H\cong (\Z/2)^{r-1}\times
(\mu_2)^r$. By Lemma \ref{finite},
the Hodge cohomology of $B(\Z/2)$ over $k$ is the cohomology
of $\Z/2$ as a group, namely the polynomial ring
$k[x]$ with $x\in H^1(B(\Z/2),\Omega^0)$.
Also, by Proposition \ref{mup}, the Hodge cohomology of $B\mu_2$
is $k[t,v]/(v^2)$ with $t\in H^1(B\mu_2,\Omega^1)$
and $v\in H^0(B\mu_2,\Omega^1)$.
Thus we have a homomorphism from $k[u_2,u_3,\ldots,u_{2r}]$
to $H^*_{\Ho}(BSO(2r)/k)$ and from there to
$H^*_{\Ho}(BH/k)\cong k[x_1,\ldots,x_{r-1},t_1,\ldots,t_r,v_1,
\ldots,v_r]/(v_i^2)$
(by the K\"unneth theorem, Proposition \ref{kunneth}).
We want to show that this composition is injective. For convenience,
we will prove the stronger statement
that the composition to 
$$H^*_{\Ho}(BH/k)/\rad=k[x_1,\ldots,x_{r-1},t_1,\ldots,t_r]$$
is injective.

We compare the restriction from $O(2r)$ to $(\Z/2)^r\times (\mu_2)^r$
with that from $SO(2r)$ to $H$:
$$\xymatrix@C-10pt@R-10pt{
k[u_1,\ldots,u_{2r}] \ar[r]\ar[d] & k[u_2,u_3,\ldots,u_{2r}]\ar[d] \\
H^*_{\Ho}(BO(2r)/k) \ar[r]\ar[d] & H^*_{\Ho}(BSO(2r)/k)\ar[d] \\
k[s_1,\ldots,s_r,t_1,\ldots,t_r]\ar[r] & k[x_1,\ldots,x_{r-1},
t_1,\ldots,t_r]
}$$
The bottom homomorphism is given (for a suitable choice of generators
$x_1,\ldots,x_{r-1}$) by $s_i\mapsto x_i$ for $1\leq i\leq r-1$
and $s_r\mapsto x_1+\cdots+x_{r-1}$ (agreeing with the fact
that $u_1\mapsto s_1+\cdots+s_r\mapsto 0$ in the Hodge cohomology
of $BH$). By the formulas for $O(2r)$, we know how the elements
$u_2,\ldots,u_{2r}$ restrict to $k[s_1,\ldots,s_r,t_1,\ldots,t_r]$,
and hence to $k[x_1,\ldots,x_{r-1},t_1,\ldots,t_r]$. Namely,
$$u_{2a}\mapsto e_a(t_1,\ldots,t_r)=\sum_{1\leq i_1<\cdots<i_a\leq r}
t_{i_1}\cdots t_{i_a},$$
and, for $1\leq a\leq r-1$,
\begin{align*}
u_{2a+1}&\mapsto \sum_{1\leq i_1<\cdots<i_a\leq r}
\bigg(\sum_{j=1}^a s_{i_j}+\sum_{m=1}^r s_m \bigg)
t_{i_1}\cdots t_{i_a}\\
&\mapsto \sum_{1\leq i_1<\cdots<i_a\leq r-1}\bigg(\sum_{j=1}^a x_{i_j}\bigg)
t_{i_1}\cdots t_{i_a}\\
&\qquad +\sum_{1\leq i_1<\cdots<i_{a-1}\leq r-1}\bigg( x_1+\cdots+x_{r-1}+
\sum_{j=1}^{a-1} x_{i_j}\bigg) 
t_{i_1}\cdots t_{i_{a-1}}t_r\\
&=\sum_{j=1}^{r-1}x_j(t_j+t_r)\sum_{\substack{1\leq i_1<\cdots<i_{a-1}
\leq r-1\\ \text{none equal to }j}}t_{i_1}\cdots t_{i_{a-1}}.
\end{align*}

We want to show that this homomorphism
$k[u_2,u_3,\ldots,u_{2r}]\arrow k[x_1,\ldots,x_{r-1},
t_1,\ldots,t_r]$ is injective. It can be factored
through $k[u_3,u_5,\ldots,u_{2r-1},t_1,\ldots,t_r]$,
by the homomorphism $\rho$
sending $u_2,u_4,\ldots,u_{2r}$ to the elementary symmetric
polynomials in $t_1,\ldots,t_r$. Since $\rho$
is injective, it remains to show that $\sigma\colon k[u_3,u_5,\ldots,u_{2r-1},
t_1,\ldots,t_r]\arrow k[x_1,\ldots,x_{r-1},t_1,\ldots,t_r]$ is injective.

As in the argument for $O(2r)$, we will show (more strongly)
that $\sigma$ is generically etale;
that is, its Jacobian determinant is not identically zero.
Because $\sigma$ is the identity on the $t_i$ coordinates,
it suffices to show that the matrix of derivatives
of $u_3,u_5,\ldots,u_{2r-1}$ with respect to $x_1,\ldots,x_{r-1}$
is nonzero for $x_1,\ldots,x_{r-1},t_1,\ldots,t_r$ generic.
This matrix of derivatives in fact only involves $t_1,\ldots,t_r$,
because $u_3,u_5,\ldots,u_{2r-1}$ have degree 1 as polynomials
in $x_1,\ldots,x_{r-1}$.
For example, for $r=3$, this $(r-1)\times (r-1)$
matrix of derivatives is
$$\begin{pmatrix} t_1+t_3 & (t_1 +t_3)(t_2)\\
t_2+t_3 & (t_2+t_3)(t_1)
\end{pmatrix},$$
where the $a$th column gives the derivatives of $u_{2a+1}$
with respect to $x_1,\ldots,x_{r-1}$. For any $r$,
the entry $(j,a)$ of the matrix (with $j,a\in\{1,\ldots,r-1\}$) is
$(t_j+t_r)e_{ja}$, where $$e_{ja}=\sum_{\substack{1\leq i_1<
\cdots <i_{a-1}\leq r-1\\ \text{none equal to }j}}t_{i_1}
\cdots t_{i_{a-1}}.$$

Since row $j$ is a multiple of $(t_j+t_r)$ for each $r$,
the determinant is $(t_1+t_r)(t_2+t_r)\cdots(t_{r-1}+t_r)$
times the determinant of the $(r-1)\times (r-1)$ matrix
$E=(e_{ja})$. So it suffices to show that the determinant
of $E$ is not identically zero. Indeed, the determinant of $E$
is the same determinant shown to be nonzero in the calculation
above for $O(2r)$, but with $r$ replaced by $r-1$.

Thus we have shown that $k[u_2,\ldots,u_n]\arrow H^*_{\Ho}(BSO(n)/k)$
is injective for $n$ even as well as for $n$ odd. We now show
that this is an isomorphism. 

Let $r=\lfloor n/2\rfloor$ and $s=\lfloor (n-1)/2\rfloor$.
Let $P$ be the parabolic subgroup of $G=SO(n)$ that stabilizes
a maximal isotropic subspace (that is, an isotropic subspace
of dimension $r$).
Then the quotient of $P$
by its unipotent radical is isomorphic to $GL(r)$.
By Proposition \ref{ls}, we have a spectral sequence
$$E_2^{ij}=H^i_{\Ho}(BG/k)\otimes H^j_{\Ho}((G/P)/k)
\imp H^{i+j}_{\Ho}(BGL(r)/k).$$
The Chow ring of $G/P$ is isomorphic to
$$\Z[e_1,\dots,e_{s}]/(e_i^2-2e_{i-1}e_{i+1}+2e_{i-2}e_{i+2}-
\cdots+(-1)^ie_{2i}),$$
where $e_i\in CH^i(G/P)$ is understood to mean zero if $i>s$
\cite[III.6.11]{MT}. (This uses Chevalley's theorem that the Chow
ring of $G/P$ for a split group $G$ is independent of the characteristic
of $k$, and is isomorphic to the integral cohomology ring
of $G_{\bf C}/P_{\bf C}$.)
By Proposition \ref{flag}, it follows that the Hodge cohomology
ring of $G/P$ is isomorphic to
$$k[e_1,\ldots,e_{s}]/(e_i^2=e_{2i}),$$
where $e_i$ is in $H^i(G/P,\Omega^i)$.
For any list of variables $x_1,\ldots,x_m$,
write $\Delta(x_1,\ldots,x_m)$ for the $k$-vector space with basis
consisting of all products $x_{i_1}\ldots x_{i_j}$
with $1\leq i_1<\cdots<i_j\leq m$
and $0\leq j\leq m$. Then we can say that
$$H^*_{\Ho}((G/P)/k)=\Delta(e_1,\ldots,e_{s}).$$

The spectral sequence converges to $H^*_{\Ho}(BGL(r)/k)=k[c_1,\ldots,c_r]$,
by Theorem \ref{integral}.
The elements $u_2,u_4,\ldots,u_{2r}$ (where $u_{2i}$ is in
$H^i(BG,\Omega^i)$)
restrict to $c_1,c_2,\ldots,c_r$.
So the $E_{\infty}$ term of the spectral sequence is concentrated
on the $0$th row and consists of the polynomial ring
$k[u_2,u_4,\ldots,u_{2r}]$.

To analyze the structure of the spectral sequence further, we use
Zeeman's comparison theorem, which he used to simplify
the proof of the Borel transgression theorem
\cite[Theorem VII.2.9]{MT}. The key point is to show
that the elements $e_i$ (possibly after adding decomposable
elements) are transgressive. (By definition, an element $u$
of $E_2^{0,q}$ in a first-quadrant spectral sequence
is {\it transgressive }if
$d_2=\cdots=d_q=0$ on $u$; then $u$ determines an element
$\tau(u):=d_{q+1}(u)$ of $E_{q+1}^{q+1,0}$, called the transgression
of $u$.)

In order to apply Zeeman's comparison theorem, we define
a model spectral sequence that maps to the spectral sequence
we want to analyze. (To be precise, we consider spectral sequences
of $k$-vector spaces, not of $k$-algebras.) As above,
let $k=\F_2$.
For a positive integer $q$,
define a spectral sequence $G_*$ with $E_2$ page given by
$G_2=\Delta(y)\otimes k[u]$,
$y$ in bidegree $(0,q)$, $u$ in bidegree $(q+1,0)$,
and $d_{q+1}(yu^j)=u^{j+1}$. 
$$\xymatrix@C-10pt@R-10pt{
k\cdot y\ar[rrd]^{\cong} & & k\cdot yu\ar[rrd]^{\cong} & & k\cdot yu^2
& \cdots\\
k\cdot 1 & & k\cdot u & & k\cdot u^2 & \cdots
}$$
Suppose that, for some positive
integer $a$, we have found elements
$y_i$ of $H^{2i}_{\Ho}((G/P)/k)$ for $1\leq i\leq a$ which are transgressive
in the spectral sequence $E_*$ above. Because $y_i$ is transgressive,
there is a map of spectral sequences $G_*\arrow E_*$
that takes the element $y$ (in degree $q=2i$)
to $y_i$. Since $E_*$ is a spectral
sequence of algebras, tensoring these maps gives a map
of spectral sequences
$$\alpha\colon
F_*:=G_*(y_1)\otimes \cdots \otimes G_*(y_a)\otimes k[u_2,u_4,\ldots,
u_{2r}]\arrow E_*.$$
(Here we are using that the elements $u_2,u_4,\ldots,u_{2r}$ are in
$H^*_{\Ho}(BG/k)$, which is row 0
of the $E_2$ page on the right, and so they are permanent cycles.)
Although we do not view the domain as a spectral sequence
of algebras, its $E_2$ page is the tensor product
of row 0 and column 0, and the map $\alpha\colon F_2\arrow E_2$
of $E_2$ pages
is the tensor product of the maps on row 0 and column 0.

Using these properties, we have
the following version of Zeeman's comparison theorem,
as sharpened by Hilton and Roitberg
\cite[Theorem VII.2.4]{MT}:

\begin{theorem}
\label{zeeman}
Let $N$ be a natural number.
Suppose that the homomorphism $\alpha\colon F_*\arrow E_*$
of spectral sequences
is bijective on $E_{\infty}^{i,j}$ for $i+j\leq N$
and injective for $i+j=N+1$, and that $\alpha$ is
bijective on row 0 of the $E_2$ page in degrees $\leq N+1$
and injective in degree $N+2$.
Then $\alpha$ is bijective on column 0 of the $E_2$
page in degree $\leq N$ and injective in degree $N+1$.
\end{theorem}

The inductive step for computing the Hodge cohomology
of $BSO(n)$ is as follows.

\begin{lemma}
\label{induction}
Let $G$ be $SO(n)$ over $k=\F_2$, $P$ the parabolic subgroup above,
$r=\lfloor n/2\rfloor$, $s=\lfloor (n-1)/2\rfloor$.
Let $N$ be a natural number, and let $a= \min(s,\lfloor N/2\rfloor)$.
Then, for each $1\leq i\leq a$, there is an element
$y_i$ in $H^i(G/P,\Omega^i)$ with the following
properties. First,
$y_i$ is equal to $e_i$ modulo
polynomials in $e_1,\ldots,e_{i-1}$ with exponents $\leq 1$.
Also, each element $y_i$ is transgressive, and any lift $v_{2i+1}$
to $H^{i+1}(BG,\Omega^i)$ of the element
$\tau(y_i)$
has the property that
$$k[u_2,u_4,\ldots,u_{2r};v_3,v_5,\ldots,
v_{2a+1}]\arrow H^*_{\Ho}(BG/k)$$
is bijective in degree $\leq N+1$
and injective in degree $N+2$. Finally, each element $v_{2i+1}$
is equal to $u_{2i+1}$ modulo polynomials in $u_2,u_3,\ldots,u_{2i}$.

More precisely, if this statement holds for $N-1$, then it holds for $N$
with the {\it same }elements $y_i$, possibly with one added.
\end{lemma}

We will apply Lemma \ref{induction}
with $N=\infty$, but the formulation with $N$ arbitrary
is convenient for the proof.

\begin{proof}
As discussed earlier, the $E_{\infty}$ page of the spectral sequence
$$E_2^{ij}=H^i_{\Ho}(BG/k)\otimes H^j_{\Ho}((G/P)/k)
\imp H^{i+j}_{\Ho}(BGL(r)/k)$$
is isomorphic to $k[u_2,u_4,\ldots,u_{2r}]$, concentrated on row 0.

We prove the lemma by induction on $N$. For $N=0$, it
is true, using that $H^0_{\Ho}(BG/k)=k$ and $H^1_{\Ho}(BG/k)=0$,
as one checks using our knowledge of the $E_{\infty}$ term.

We now assume the result for $N-1$, and prove it for $N$.
By the inductive assumption, for $b:=\min(s,\lfloor (N-1)/2\rfloor)$,
we can choose
$y_1,\ldots,y_b$ such that $y_i\in H^i(G/P,\Omega^i)$ is equal
to $e_i$ modulo
polynomials in $e_1,\ldots,e_{i-1}$ with exponents $\leq 1$,
$y_i$ is transgressive
for the spectral sequence, and, if we define $v_{2i+1}\in
H^{i+1}(BG,\Omega^i)$
to be any lift (from the $E_{2i+1}$ page to the $E_2$ page)
of the transgression $\tau(y_i)$
for $1\leq i\leq b$, the homomorphism
$$k[u_2,u_4,\ldots,u_{2r};v_3,v_5,\ldots,v_{2b+1}]
\arrow H^*_{\Ho}(BG/k)$$
is bijective in degree $\leq N$
and injective in degree $N+1$.
Finally,
the element $v_{2i+1}$ for $1\leq i\leq b$ is equal
to $u_{2i+1}$ modulo polynomials in $u_2,u_3,\ldots,u_{2i}$.

Also, by the injectivity in degree $N+1$ (above),
it follows that there is a set (possibly empty) of elements $z_i$
in $H^{N+1}_{\Ho}(BG/k)$ such that
$$\varphi\colon k[u_2,u_4,\ldots,u_{2r};
v_3,v_5,\ldots,v_{2b+1};z_i]\arrow H^*_{\Ho}(BG/k)$$
is bijective in degrees at most $N+1$. (Recall that
$b=\min(s,\lfloor (N-1)/2\rfloor)$.) The elements $z_i$
do not affect the domain of $\varphi$ in degree $N+2$
(because that ring is zero in degree 1). Therefore, $\varphi$
is injective in degree $N+2$, because
$$k[u_2,u_4,\ldots,u_{2r}; v_3,v_5,\ldots,
v_{2b+1}]\arrow H^*_{\Ho}(BG/k)$$
is injective. (This uses that $v_{2i+1}$ is equal to $u_{2i+1}$
modulo polynomials in $u_2,u_3,\ldots,u_{2i}$, together with
the injectivity of $k[u_2,u_3,\ldots,u_n]\arrow H^*_{\Ho}(BG/k)$,
shown earlier.)

The elements $z_i$ can be chosen to become zero in the $E_{\infty}$ page,
because the $E_{\infty}$ page is just $k[u_2,u_4,\ldots,
u_{2r}]$ on row 0. Therefore,
there are transgressive
elements $w_i\in H^N_{\Ho}((G/P)/k)$
with $z_i=\tau(w_i)$ in the $E_{N+1}$ page.
(If $z_i$ is killed before $E_{N+1}$,
we can simply take $w_i=0$.) By Zeeman's comparison theorem
(Theorem \ref{zeeman}),
the homomorphism
$$\psi\colon \Delta(y_1,\ldots,y_b;w_i)\arrow H^*_{\Ho}((G/P)/k)$$
is bijective in degrees $\leq N$ and injective in degree $N+1$.

Let $a=\min(s,\lfloor N/2\rfloor)$. 
We know that $\Delta(e_1,\ldots,e_a)\arrow H^*_{\Ho}((G/P)/k)$
is bijective in degrees $\leq N$. Since the elements $w_i$
are in degree $N$, while $b=\min(s,\lfloor (N-1)/2\rfloor)$,
we deduce that there
is no element $w_i$ if $N$ is odd or $N>2s$, and there is exactly one $w_i$
if $N$ is even and $N\leq 2s$. In the latter case, we have
$a=N/2$; in that case,
let $y_a$ denote the single element
$w_i$. Since we know that $H^*_{\Ho}((G/P)/k)=\Delta(e_1,\ldots,e_s)$,
$y_a$ must be equal
to $e_a$ modulo polynomials in $e_1,\ldots,e_{a-1}$
with exponents $\leq 1$. By construction, $y_a$ is transgressive.
Also,
in the case where $N$ is even and $N\leq 2s$,
let $v_{2a+1}$ in $H^{a+1}(BG,\Omega^a)$
be a lift to the $E_2$ page of the element
$\tau(y_a)$ (formerly called $z_i$).
Then we know that
$$\varphi\colon k[u_2,u_4,\ldots,u_{2r};
v_3,v_5,\ldots,v_{2a+1}]\arrow H^*_{\Ho}(BG/k)$$
is bijective in degree $\leq N+1$.
In the case where $N$ is even and $N\leq 2s$ (where we have added
one element $v_{2a+1}$ to those constructed before), this bijectivity
in degree $N+1=2a+1$ together with the injectivity
of $k[u_2,u_3,\ldots,u_n]\arrow H^*_{\Ho}(BG/k)$ in all degrees
implies that $v_{2a+1}$ must be equal to $u_{2a+1}$ modulo
polynomials in $u_2,u_3,\ldots,u_{2a}$. By the same injectivity,
it follows that $\varphi$ is injective in degree $N+2$.
\end{proof}

We can take $N=\infty$ in Lemma \ref{induction},
because the elements $y_1,\ldots,y_s$
do not change as we increase $N$. This gives
that $k[u_2,u_3,\ldots,u_n]\arrow H^*_{\Ho}(BSO(n)/k)$
is an isomorphism. (The element $v_{2i+1}$ produced by Lemma \ref{induction}
need not be the element $u_{2i+1}$ defined earlier, but $v_{2i+1}$
is equal to $u_{2i+1}$ modulo decomposable elements,
which gives this conclusion.)

Using the Hodge cohomology of $BSO(2r)$, we can compute
the Hodge cohomology of $BO(2r)$ over $k$ using
the Hochschild-Serre spectral
sequence of Lemma \ref{pi0}:
$$E_2^{ij}=H^i(\Z/2,H^j(BSO(2r),\Omega^*))\imp H^{i+j}(BO(2r),\Omega^*).$$
We have a homomorphism $k[u_1,u_2,\ldots,u_{2r}]\arrow BO(2r)$
whose composition to $BSO(2r)$ is surjective. Therefore, $\Z/2$
acts trivially on the Hodge cohomology of $BSO(2r)$, and all differentials
are zero on column 0 of this spectral sequence. It follows that
the spectral sequence degenerates at $E_2$, and hence
\begin{align*}
H^*_{\Ho}(BO(2r)/k)&\cong H^*(\Z/2,k)\otimes H^*_{\Ho}(BSO(2r)/k)\\
&\cong k[u_1,u_2,\ldots,u_{2r}].
\end{align*}

Finally, we show that the Hodge spectral sequence
$$E_1^{ij}=H^j(BG,\Omega^i)\imp H^{i+j}_{\dR}(BG/k)$$
degenerates for $G=SO(n)$ over $k$.
Indeed, by restricting to a maximal torus $T=(G_m)^r$
of $G$, the elements
$u_2,u_4,\ldots,u_{2r}$ restrict to the elementary
symmetric polynomials in the generators of $H^*_{\dR}(BT/k)
=k[t_1,\ldots,t_r]$. Therefore, the ring $k[u_2,u_4,\ldots,u_{2r}]$
injects into $H^*_{\dR}(BG/k)$. So
all differentials into the main diagonal $\oplus_i H^{i,i}$
of the Hodge spectral sequence for $BG$ are zero.
$$\xymatrix@C-10pt@R-10pt{
H^2(BG,\Omega^0)\ar[r]^{d_1}\ar@{-->}[rrd]^{d_2} & H^2(BG,\Omega^1)\ar[r]^{d_1}
  & H^2(BG,\Omega^2)\\
H^1(BG,\Omega^0)\ar[r]^{d_1}\ar@{-->}[rrd]^{d_2} & H^1(BG,\Omega^1)\ar[r]
  & 0\\
H^0(BG,\Omega^0)\ar[r] & 0\ar[r] & 0
}$$
It follows that all differentials are zero on the elements
$u_{2i+1}\in H^{i+1}(BG,\Omega^i)$: only $d_1$ maps $u_{2i+1}$
into a nonzero group, and that is on the main diagonal.
Also, all differentials are zero on the elements $u_{2i}$
in the main diagonal (since they map into zero groups).
This proves the degeneration of the Hodge spectral sequence.
Therefore, $H^*_{\dR}(BSO(n)/k)$ is isomorphic to $k[u_2,u_3,\ldots,u_n]$.

The same argument proves the degeneration of the Hodge spectral
sequence for $BO(2r)$. Therefore, $H^*_{\dR}(BO(2r)/k)$
is isomorphic to $k[u_1,u_2,\ldots,u_{2r}]$.

Finally, $O(2r+1)$ is isomorphic to $SO(2r+1)\times \mu_2$,
and so the calculation for $BO(2r+1)$ follows from those for
$BSO(2r+1)$ (above) and $B\mu_2$ (Proposition \ref{mup}),
by the K\"unneth theorem (Proposition \ref{kunneth}).
Theorem \ref{son} is proved.
\end{proof}

\begin{proof}
(Proposition \ref{whitney}) Let $2r$ and $2s$ be the ranks of the quadratic
bundles $E$ and $F$. The problem amounts to computing the restriction
from $BO(2r+2s)$ to $BO(2r)\times BO(2s)$ on Hodge cohomology
or de Rham cohomology.
We first compute $u(E\oplus F)$
in Hodge cohomology. The formula for $u_{2a}(E\oplus F)$ follows
from the definition of $u_{2a}$ in $H^a(BO(2r+2s),\Omega^a)$. (Since
$u_{2a}$ is in $H^a(BO(2r+2s),\Omega^a)$, its restriction
to the Hodge cohomology of $BO(2r)\times BO(2s)$ must be in
$H^a(BO(2r)\times BO(2s),\Omega^a)$, which explains why only the even
$u$-classes of $E$ and $F$ appear in the formula.) The formula
for $u_{2a+1}(E\oplus F)$ follows from the formula for
$u_{2a}(E\oplus F)$, using that $u_{2a+1}=\beta u_{2a}+u_1u_{2a}$.

In de Rham cohomology, the same formulas hold for $u(E\oplus F)$.
This uses that for any affine $k$-group scheme $G$, since $H^i(BG,\Omega^j)=0$
for $i<j$ by Theorem \ref{iso}, the subring $\oplus_i H^i(BG,\Omega^i)$
of Hodge cohomology canonically maps into de Rham
cohomology.
\end{proof}

\section{The spin groups}

In contrast to the other calculations in this paper,
we now exhibit a reductive group $G$
such that the mod 2 cohomology of the topological space $BG_{\C}$
is not isomorphic
to the de Rham cohomology of the algebraic stack $BG_{\F_2}$,
even additively. The example was suggested by the observation
of Feshbach, Benson, and Wood that the restriction
$H^*(BG_{\C},\Z)\arrow H^*(BT_{\C},\Z)^W$ fails to be surjective
for $G=\Spin(n)$ if $n\geq 11$ and $n\equiv 3,4,5\pmod{8}$
\cite{BW}. For simplicity, we work out the case of $\Spin(11)$.
It would be interesting to make a full computation
of the de Rham cohomology of $B\Spin(n)$ in characteristic 2.

\begin{theorem}
\label{spin}
$$\dim_{\F_2} H^{32}_{\dR}(B\Spin(11)/\F_2)
> \dim_{\F_2} H^{32}(B\Spin(11)_{\C},\F_2).$$
\end{theorem}

\begin{proof}
Let $k=\F_2$. Let $n$ be an integer at least 6; eventually, we will restrict
to the case $n=11$.
Let $G$ be the split group $\Spin(n)$ over $k$, and let $T$
be a maximal torus in $G$. Let $r=\lfloor n/2\rfloor$. The Weyl
group $W$ of $G$ is $S_r\ltimes (\Z/2)^{r}$ for $n=2r+1$,
and the subgroup $S_r\ltimes (\Z/2)^{r-1}$ for $n=2r$.
We start by computing the ring
$O(\t)^W$ of $W$-invariant functions on the Lie algebra $\t$ of $T$.

First consider the easier case where $n$ is odd, $n=2r+1$.
The element $-1$ in $(\Z/2)^r\subset W$ acts as the identity
on $\t$, since we are in characteristic 2. The ring
$O(\t)^W$ can also be viewed as
$S(X^*(T)\otimes k)^W$.
Computing this ring is similar to, but simpler than, Benson and Wood's
calculation of $S(X^*(T))^W=H^*(BT_{\C},\Z)^W$ \cite{BW}. We follow their
notation.

We have
$$S(X^*(T))\cong \Z[x_1,\ldots,x_r,A]/(2A=x_1+\cdots+x_r),$$
by thinking of $T$ as the double cover of a maximal torus
in $SO(2r+1)$. The symmetric group $S_r$ in $W$ permutes $x_1,\ldots,x_r$
and fixes $A$. The elementary abelian group $E_r=(\Z/2)^r$ in $W$, with
generators $\epsilon_1,\ldots,\epsilon_r$, acts by: $\epsilon_i$
changes the sign of $x_i$ and fixes $x_j$ for $j\neq i$,
and $\epsilon_i(A)=A-x_i$. So
$$S(X^*(T)\otimes k)\cong k[x_1,\ldots,x_r,A]/(x_1+\cdots+x_r).$$
Note that $-1:=\epsilon_1\cdots\epsilon_r$ in $W$ acts as the identity
on $S^*(X^*(T)\otimes k)$.

We first compute the invariants of the subgroup $E_r$
on $S(X^*(T)\otimes k)$,
using the following lemma.

\begin{lemma}
\label{inv2}
Let $R$ be an $\F_2$-algebra which is a domain,
$S$ the polynomial ring $R[x]$,
and $a$ a nonzero element of $R$. Let $G=\Z/2$ act on $S$
by fixing $R$ and sending $x$ to $x+a$. Then the ring of invariants is
$$S^G=R[u],$$
where $u=x(x+a)$.
\end{lemma}

\begin{proof}
Clearly $u=x(x+a)$ in $S$ is $G$-invariant. Since $u$ is a monic
polynomial of degree 2 in $x$, we have $S=R[u]\oplus x\cdot R[u]$.
Let $\sigma$ be the generator of $G=\Z/2$. Any element
of $S$ can be written as $f+xg$ for some (unique) elements
$f,g\in R[u]$. If $f+xg$ is $G$-invariant, then
$0=\sigma(f+xg)-(f+xg)=(x+a)g-xg=ag$. Since $a$ is a non-zero-divisor
in $R$, it is a non-zero-divisor in $R[u]$; so $g=0$. Thus
$S^G=R[u]$.
\end{proof}

Let $E_j\cong (\Z/2)^j$ be the subgroup of $W$ generated by $\epsilon_1,
\ldots,\epsilon_j$. Let
$$\eta_j=\prod_{I\subset \{1,\ldots,j\} }\bigg( A-\sum_{i\in I}x_i\bigg) ,$$
which is $E_j$-invariant. Here $\eta_j$ has degree $2^{j}$
in $S^*(X^*(T)\otimes k)$. By Lemma \ref{inv2} (with $R=k[x_1,
\ldots,x_r]/(x_1+\cdots+x_r)$) and induction on $j$,
we have
$$S^*(X^*(T)\otimes k)^{E_j}=k[x_1,\ldots,x_r,\eta_j]/(x_1+\cdots+x_r=0)$$
for $1\leq j\leq r-1$. Since $-1=\epsilon_1\cdots\epsilon_r$
acts as the identity on these rings, we also have
$$S^*(X^*(T)\otimes k)^{E_r}
=k[x_1,\ldots,x_r,\eta_{r-1}]/(x_1+\cdots+x_r=0).$$

The symmetric group $S_r$ permutes $x_1,\ldots,x_r$,
and it fixes $\eta_{r-1}$. Therefore, computing the invariants
of the Weyl group on $S^*(X^*(T)\otimes k)$ reduces to computing
the invariants of the symmetric group $S_r$ on
$R=k[x_1,\ldots,x_r]/(x_1+\cdots+x_r)$. Write $c_1,\ldots,c_r$
for the elementary symmetric polynomials in $k[x_1,\ldots,x_r]$.
For $r\geq 3$, the ring of invariants $R^{S_r}$
is equal to $k[c_1,\ldots,c_r]/(c_1)=
k[c_2,\ldots,c_r]$ \cite[Proposition 4.1]{Nakajima}.

The answer is different for $r=2$: then
$S_2$ acts trivially on $R=k[x_1,x_2]/(x_1+x_2)$,
and so $R^{S_2}=R=k[x_1]$.

Combining these calculations with the earlier ones,
we have found the invariants for the Weyl group $W$ of $G=\Spin(2r+1)$:
for $r\geq 1$,
$$S^*(X^*(T)\otimes k)^W=\begin{cases} k[c_2,\ldots,c_r,\eta_{r-1}]
&\text{if }r\neq 2,\\
k[x_1,\eta_1] & \text{if }r=2.
\end{cases}$$
Here $|c_i|=i$ for $2\leq i\leq r$, $|x_1|=2$,
and $|\eta_{r-1}|=2^{r-1}$.

We now compute $S^*(X^*(T)\otimes k)^W$ for $G=\Spin(2r)$. Note that a maximal
torus in $\Spin(2r)$ is also a maximal torus in $\Spin(2r+1)$. So we have
again
$$S^*(X^*(T)\otimes k)\cong k[x_1,\ldots,x_r,A]/(x_1+\cdots+x_r).$$
The Weyl group $W=S_r\ltimes (\Z/2)^{r-1}$ acts on this ring by: $S_r$
permutes $x_1,\ldots,x_r$, and fixed $A$, and $(\Z/2)^{r-1}$ is the subgroup
$\langle \epsilon_1\epsilon_2,\ldots,\epsilon_1\epsilon_r\rangle$
in the notation above. Thus $\epsilon_1\epsilon_j$ fixes each $x_j$
(since we are working modulo 2) and sends $A$ to $A-x_1-x_j$.

For $1\leq j\leq r$, let $F_j$ be the subgroup $\langle \epsilon_1\epsilon_2,
\ldots,\epsilon_1\epsilon_j\rangle\cong (\Z/2)^{j-1}\subset W$.
Let
$$\mu_j=\prod_{\substack{I\subset \{1,\ldots,j\}\\ |I|\text{ even}}}
\bigg( A-\sum_{i\in I}x_i\bigg) .$$
Then $|\mu_j|=2^{j-1}$ and $\mu_1=A$. Clearly $\mu_j$ is $F_j$-invariant.
Benson and Wood observed (or one can check directly)
that if $r$ is even and $r\geq 4$, then $\mu_{r-1}$ is in
fact $W$-invariant, while if $r$ is odd and $r\geq 3$, then $\mu_r$
is $W$-invariant \cite[Proposition 4.1]{BW}.

For $1\leq j\leq r-1$, an induction on $j$ using Lemma \ref{inv2}
gives that
$$S^*(X^*(T)\otimes k)^{F_j}=k[x_1,\ldots,x_r,\mu_j]/(x_1+\cdots+x_r).$$
If $r$ is even, then $-1:=\epsilon_1\cdots \epsilon_r$ is in $F_r\subset W$,
and it acts trivially on $S^*(X^*(T)\otimes k)$.
Therefore, for $r$ even, we have
$$S^*(X^*(T)\otimes k)^{F_r}=k[x_1,\ldots,x_r,\mu_{r-1}]/(x_1+\cdots+x_r).$$
If $r$ is odd, then we can apply Lemma \ref{inv2} one more time,
yielding that
$$S^*(X^*(T)\otimes k)^{F_r}=k[x_1,\ldots,x_r,\mu_{r}]/(x_1+\cdots+x_r).$$

The subgroup $S_r\subset W$ permutes $x_1,\ldots,x_r$, and fixes
$\mu_{r-1}$, resp.\ $\mu_r$. We showed above that
$$k[x_1,\ldots,x_r]/(x_1+\cdots+x_r)^{S_r}=k[c_2,\ldots,c_r].$$
Therefore, for $G=\Spin(2r)$, we have
$$S^*(X^*(T)\otimes k)^W=\begin{cases} k[c_2,\ldots,c_r,\mu_{r-1}]
&\text{if $r$ is even}\\
k[c_2,\ldots,c_r,\mu_{r}] &\text{if $r$ is odd.}
\end{cases}$$
Here $|c_i|=i$ for $2\leq i\leq r$ and $|\mu_{r-1}|=2^{r-2}$,
resp.\ $|\mu_r|=2^{r-1}$.

Thus we have determined $S^*(X^*(T)\otimes k)^W$ for $G=\Spin(n)$ for all $n$,
even or odd. Now think of $G=\Spin(n)$
as a split reductive group over $k$.
By Theorem \ref{invariant}, the ring $S^*(X^*(T)\otimes k)^W=O(\t)^W$
can be identified with $O(\g)^G$
for all $n\geq 6$.
(The exceptional cases $\Spin(3), \Spin(4), \Spin(5)$
are the spin groups that have a factor isomorphic to a symplectic group:
$\Spin(3)\cong \Sp(2)$, $\Spin(4)\cong \Sp(2)\times \Sp(2)$,
and $\Spin(5)\cong \Sp(4)$.)
We deduce that for $n\geq 6$,
$$O(\g)^G=\begin{cases} k[c_2,\ldots,c_r,\eta_{r-1}]
&\text{if }n=2r+1\\
k[c_2,\ldots,c_r,\mu_{r-1}]
&\text{if }n=2r\text{ and $r$ is even}\\
k[c_2,\ldots,c_r,\mu_{r}]
&\text{if }n=2r\text{ and $r$ is odd}.
\end{cases}$$

For $G=\Spin(n)$ and any $n\geq 6$, we have homomorphisms
$$O(\g)^G\arrow H^*_{\dR}(BG/k)\arrow H^*_{\dR}(BT/k)^W=O(\t)^W,$$
whose composition is the obvious inclusion. (The first homomorphism
comes from the isomorphism of $O(\g)^G$ with $\oplus_i H^i(BG_{k},
\Omega^i)$, using that $H^i(BG_{k},\Omega^j)=0$ for $i<j$.)
In this case, the restriction $O(\g)^G\arrow O(\t)^W$ is a bijection.
So $H^*_{\dR}(BG/k)$ contains the ring computed above (with degrees
multiplied by 2), and retracts onto
it. It follows that for all $n\geq 6$, $H^*_{\dR}(BG/k)$ has
an indecomposable generator in degree $2^r$ if $n=2r+1$,
in degree $2^{r-1}$ if $n=2r$ and $r$ is even, and in degree
$2^r$ if $n=2r$ and $r$ is odd. (For this argument, we do not need
to find all the indecomposable generators of $H^*_{\dR}(BG/k)$.)

Compare this with Quillen's calculation of the cohomology
of the classifying space
of the complex reductive group $\Spin(n)_{\C}$, or equivalently
of the compact Lie group $\Spin(n)$
\cite[Theorem 6.5]{Quillenspin}:
$$H^*(B\Spin(n)_{\C},k)\cong H^*(BSO(n)_{\C},k)/J\otimes k[w_{2^h}
(\Delta_{\theta})].$$
Here $\Delta_{\theta}$ is a faithful orthogonal representation
of $\Spin(n)_{\C}$ of minimal dimension,
and $J$ is the ideal generated by the regular sequence
$$w_2,\Sq^1 w_2,\ldots,\Sq^{2^{h-2}}\cdots\Sq^2\Sq^1 w_2$$
in the polynomial ring $H^*(BSO(n)_{\C},k)=k[w_2,w_3,\ldots,
w_n]$, where $|w_i|=i$. Finally, the number $h$ is
given by the following table:

$$\begin{array}{cc}
n & h\\
8l+1 & 4l+0\\
8l+2 & 4l+1\\
8l+3 & 4l+2\\
8l+4 & 4l+2\\
8l+5 & 4l+3\\
8l+6 & 4l+3\\
8l+7 & 4l+3\\
8l+8 & 4l+3
\end{array}$$

The Steenrod operations on the mod 2 cohomology of $BSO(n)_{\C}$,
as used in the formula above, are known, by Wu's formula
\cite[Theorem III.5.12]{MT}:
$$\Sq^i w_j=\sum_{l=0}^i \binom{j-l-1}{i-l}w_lw_{i+j-l}$$
for $0\leq i\leq j$, where by convention $\binom{-1}{0}=1$.

Write $r=\lfloor n/2\rfloor$. If $n=2r+1$, then
the generator $w_{2^h}(\Delta_{\theta})$ is in degree
$2^r$ if $r\equiv 0,3\pmod{4}$ and in degree $2^{r+1}$
if $r\equiv 1,2\pmod{4}$. If $n=2r$, then the generator
$w_{2^h}(\Delta_{\theta})$ is in degree $2^{r-1}$ if $r\equiv 0\pmod{4}$
and in degree $2^r$ if $r\equiv 1,2,3\pmod{4}$. Therefore, for $n\geq 11$,
$H^*(B\Spin(n)_{\C},k)$ has no indecomposable
generator in degree $2^r$ if $n\equiv 3,5\pmod{8}$, and no indecomposable
generator in degree $2^{r-1}$ if $n\equiv 4\pmod{8}$. But $H^*_{\dR}(BG/k)$
does have an indecomposable generator in the indicated degree $2^a$, as shown
above. Thus, for $G=\Spin(n)$, $H^*(BG_{\C},k)$ is not isomorphic
to $H^*_{\dR}(BG/k)$ as a graded ring when $n\geq 11$
and $n\equiv 3,4,5\pmod{8}$.

We want to show, more precisely, that for $n=11$,
$H^{32}_{\dR}(BG/k)$ is bigger
than $H^{32}(BG_{\C},k)$. We know
the cohomology of $BG_{\C}$ by Quillen (above), and so it remains to give
a lower bound for the de Rham cohomology of $BG$ over $k$.

We do this by restricting to a suitable
abelian $k$-subgroup scheme of $G=\Spin(n)$. Assume that
$n\not\equiv 2\pmod{4}$; this includes the case $\Spin(11)$ that we
are aiming for. Then the Weyl group $W$ of $\Spin(n)$ contains $-1$.
So $\Spin(n)$ contains an extension of $\Z/2$
by a split maximal torus $T\cong (G_m)^r$,
where $\Z/2$ acts by inversion on $T$. Let $L$ be the subgroup
of the form $1\arrow T[2]\arrow L\arrow \Z/2\arrow 1$; then $L$ is abelian
(because inversion is the identity on $T[2]\cong (\mu_2)^r$). Since
the field $k=\F_2$ is perfect, the reduced locus of $L$ is a $k$-subgroup
scheme (isomorphic to $\Z/2$) \cite[Corollary 1.39]{Milne},
and so the extension splits. That is,
$L\cong (\mu_2)^r\times \Z/2$.

Let us compute the pullbacks of the generators $u_i$
of $H^*_{\dR}(BSO(n)/k)$ (Theorem \ref{son})
to the subgroup $L$ of $G=\Spin(n)$.
It suffices to compute the restrictions of the classes $u_i$
to the image $K$ of $L$ in $SO(n)$; clearly $K\cong (\mu_2)^{r-1}\times
\Z/2$. In notation similar to that
used earlier in this proof, the ring of polynomial
functions on the Lie algebra of the subgroup $(\mu_2)^{r-1}$ here is
$$k[t_1,\ldots,t_r]/(t_1+\cdots+t_r).$$
This ring can be viewed as the Hodge cohomology ring of $B(\mu_2)^{r-1}$
modulo its radical,
with the generators $t_i$ in $H^1(B(\mu_2)^{r-1},\Omega^1)$
(by Propositions \ref{mup}
and \ref{kunneth}). Using Lemma \ref{finite}, we conclude that
$$H^*_{\Ho}(BK/k)/\rad\cong k[s,t_1,\ldots,t_r]/(t_1+\cdots+t_r),$$
where $s$ is pulled back from the generator of $H^1(B(\Z/2),O)$.
The Hodge spectral sequence for $BK$ degenerates at $E_1$,
since we know this degeneration for
$B\Z/2$ and $B(\mu_2)^{r-1}$.
Therefore,
$$H^*_{\dR}(BK/k)/\rad\cong k[s,t_1,\ldots,t_r]/(t_1+\cdots+t_r),$$

Note that the surjection $L\arrow K$ is split. So
if we compute that an element
of $H^*_{\dR}(BSO(n)/k)$ has nonzero restriction to $K$, then it
has nonzero restriction to $L$, hence a fortiori to $G=\Spin(n)$.

Now strengthen the assumption $n\not\equiv 2\pmod{4}$ to
assume that $n$ is odd and $n\geq 7$.
In the proof of Theorem \ref{son}, we computed the restriction
of $u_2,u_3,\ldots,u_{2r+1}$ from $SO(2r+1)$ to its subgroup
$O(2)^r$, and hence to its subgroup $(\mu_2)^r\times (\Z/2)^r$.
(We worked there in Hodge cohomology, but the formulas remain
true in de Rham cohomology.)
We now want to restrict to the smaller subgroup $K=(\mu_2)^{r-1}
\times \Z/2$. This last step sends $H^*_{\dR}(B((\mu_2)^r\times
(\Z/2)^r)/k)/\rad=k[s_1,\ldots,s_r,t_1,\ldots,t_r]$
to $H^*_{\dR}(BK/k)/\rad
=k[s,t_1,\ldots,t_r]/(t_1+\cdots+t_r)$ by $s_i\mapsto s$ for all $i$
and $t_i\mapsto t_i$. By the formulas from the proof of Theorem
\ref{son}, the element $u_{2a}$ (for $1\leq a\leq r$)
restricts to the elementary symmetric polynomial
$$c_a=\sum_{1\leq i_1<\cdots<i_a\leq r}
t_{i_1}\cdots t_{i_a}.$$
Thus $u_2$ restricts to 0 on $K$, but $u_4,u_6,\ldots,u_{2r}$ restrict
to generators of the polynomial ring
$$(k[t_1,\ldots,t_r]/(t_1+\cdots+t_r))^{S_r}\subset H^*_{\dR}(BK/k)/\rad,$$
using that $r\geq 3$, as discussed earlier in this section.

Next, using notation from the proof of Theorem \ref{son},
for $1\leq a\leq r$,
the restriction of $u_{2a+1}$ to $H^*_{\dR}(BK/k)/\rad$
is (first restricting
from $SO(2r+1)$ to its subgroup $(\mu_2)^r\times (\Z/2)^r$,
and then to $K=(\mu_2)^{r-1}\times \Z/2$):
\begin{align*}
u_{2a+1}&\mapsto \sum_{1\leq i_1<\cdots<i_a\leq r}\bigg( \sum_{j=1}^a
s_{i_j} \bigg) t_{i_1}\cdots t_{i_a}\\
&\mapsto asu_{2a}.
\end{align*}
Thus, for all $1\leq a\leq r$,
$u_{2a+1}$ restricts in $H^*_{\dR}(BK/k)/\rad$ to $su_{2a}$ if $a$ is odd,
and otherwise to zero. (But $u_2$ restricts to 0, and so this also means
that $u_3$ restricts to 0.)

This gives a lower bound for the image of $H^*_{\dR}(BSO(n)/k)
\arrow H^*_{\dR}(B\Spin(n)/k)$ for $n$ odd. In particular,
for $n=11$, this image has Hilbert series at least that of
the ring
$$k[u_4,u_6,u_7,u_8,u_{10},u_{11}]/(u_{11}u_6+u_{10}u_7),$$
since the latter ring is isomorphic to the image
of restriction from $SO(11)$ to $H^*_{\dR}(BL/k)/\rad$,
where $L\subset \Spin(11)$.

We now compare this to Quillen's computation (above)
in the case of $\Spin(11)$:
\begin{multline*}
H^*(B\Spin(11)_{\C},k)=k[w_4,w_6,w_7,w_8,w_{10},w_{11},
w_{64}(\Delta_{\theta})]/(w_{11}w_6+w_{10}w_7,\\
w_{11}^3+w_{11}^2w_7w_4+w_{11}w_8w_7).
\end{multline*}
Since the last generator $w_{64}(\Delta_{\theta})$
is in degree 64 and the last
relation is in degree 33, the degree-32 component
of this ring has the same dimension as the degree-32 component
of the lower bound above for $H^*_{\dR}(B\Spin(11)/k)$.
However, earlier in this section,
we showed that $H^*_{\dR}(B\Spin(11)/k)$ has an extra generator
$\mu_5$ in degree 32. This is linearly independent of
the image of restriction from $SO(11)$, as we see by restricting
to a maximal torus $T$ in $\Spin(11)$. Indeed, we showed earlier in this
section that the image of $H^*_{\dR}(B\Spin(11)/k)\arrow H^*_{\dR}(BT/k)$
is the polynomial ring
$k[c_2,\ldots,c_5,\mu_5]$,
whereas the image of the pullback from $SO(11)$ to
$T\subset \Spin(11)$ is just $k[c_2,\ldots,
c_5]$ ($=k[w_4,w_6,w_8,w_{10}]$). Thus we have shown that
$$\dim_k H^{32}_{\dR}(B\Spin(11)/k)>\dim_k H^{32}(B\Spin(11)_{\C},k).$$
\end{proof}

% Omit these bibliography lines if there's no bibliography.

\small \sc UCLA Mathematics Department, Box 951555,
Los Angeles, CA 90095-1555

totaro@math.ucla.edu
\end{document}